\documentclass{article}
\usepackage{amsmath}
\usepackage{amssymb}
\usepackage{amsthm}
\usepackage{cite}
\usepackage[arrow,curve,matrix]{xy}
\begin{document}
\def\e#1\e{\begin{equation}#1\end{equation}}
\def\ea#1\ea{\begin{align}#1\end{align}}
\def\eq#1{{\rm(\ref{#1})}}
\theoremstyle{plain}
\newtheorem{thm}{Theorem}[section]
\newtheorem{lem}[thm]{Lemma}
\newtheorem{prop}[thm]{Proposition}
\newtheorem{cor}[thm]{Corollary}
\theoremstyle{definition}
\newtheorem{dfn}[thm]{Definition}
\newtheorem{ass}[thm]{Assumption}
\newtheorem{quest}[thm]{Question}
\newtheorem{rem}[thm]{Remark}
\newtheorem{ex}[thm]{Example}
\def\na{{\rm na}}
\def\stk{{\rm stk}}
\def\ind{{\rm \kern.05em ind}}
\def\fin{{\rm fin}}
\def\Ho{{\rm Ho}}
\def\Po{{\rm Po}}
\def\uni{{\rm uni}}
\def\rk{{\rm rk}}
\def\vi{{\rm vi}}
\def\opp{{\rm op}}
\def\dim{\mathop{\rm dim}\nolimits}
\def\cha{\mathop{\rm char}}
\def\Im{\mathop{\rm Im}}
\def\GL{\mathop{\rm GL}\nolimits}
\def\Spec{\mathop{\rm Spec}\nolimits}
\def\Stab{\mathop{\rm Stab}\nolimits}
\def\Sch{\mathop{\rm Sch}\nolimits}
\def\coh{\mathop{\rm coh}}
\def\Hom{\mathop{\rm Hom}\nolimits}
\def\Iso{\mathop{\rm Iso}\nolimits}
\def\Aut{\mathop{\rm Aut}}
\def\End{\mathop{\rm End}}
\def\ad{\mathop{\rm ad}}
\def\Mor{\mathop{\rm Mor}\nolimits}
\def\CF{\mathop{\rm CF}\nolimits}
\def\CFi{\mathop{\rm CF}\nolimits^{\rm ind}}
\def\LCF{\mathop{\rm LCF}\nolimits}
\def\dLCF{{\dot{\rm LCF}}\kern-.1em\mathop{}\nolimits}
\def\SF{\mathop{\rm SF}\nolimits}
\def\SFa{\mathop{\rm SF}\nolimits_{\rm al}}
\def\SFai{\mathop{\rm SF}\nolimits_{\rm al}^{\rm ind}}
\def\uSF{\mathop{\smash{\underline{\rm SF\!}\,}}\nolimits}
\def\uSFi{\mathop{\smash{\underline{\rm SF\!}\,}}\nolimits^{\rm ind}}
\def\oSF{\mathop{\bar{\rm SF}}\nolimits}
\def\oSFa{\mathop{\bar{\rm SF}}\nolimits_{\rm al}}
\def\oSFai{{\ts\bar{\rm SF}{}_{\rm al}^{\rm ind}}}
\def\uoSF{\mathop{\bar{\underline{\rm SF\!}\,}}\nolimits}
\def\uoSFa{\mathop{\bar{\underline{\rm SF\!}\,}}\nolimits_{\rm al}}
\def\uoSFi{\mathop{\bar{\underline{\rm SF\!}\,}}\nolimits_{\rm al}^{\rm
ind}}
\def\LSF{\mathop{\rm LSF}\nolimits}
\def\LSFa{\mathop{\rm LSF}\nolimits_{\rm al}}
\def\uLSF{\mathop{\smash{\underline{\rm LSF\!}\,}}\nolimits}
\def\dLSF{{\dot{\rm LSF}}\kern-.1em\mathop{}\nolimits}
\def\doLSF{{\dot{\bar{\rm LSF}}}\kern-.1em\mathop{}\nolimits}
\def\duoLSF{{\dot{\bar{\underline{\rm LSF\!}\,}}}\kern-.1em\mathop{}
\nolimits}
\def\ouLSF{{\bar{\underline{\rm LSF\!}\,}}\kern-.1em\mathop{}
\nolimits}
\def\dLSFi{{\dot{\rm LSF}}\kern-.1em\mathop{}\nolimits^{\rm ind}}
\def\dLSFa{{\dot{\rm LSF}}\kern-.1em\mathop{}\nolimits_{\rm al}}
\def\doLSFa{{\dot{\bar{\rm LSF}}}\kern-.1em\mathop{}\nolimits_{\rm al}}
\def\dLSFai{{\dot{\rm LSF}}\kern-.1em\mathop{}\nolimits^{\rm ind}_{\rm
al}}
\def\duLSF{{\dot{\underline{\rm LSF\!}\,}}\kern-.1em\mathop{}\nolimits}
\def\duLSFi{{\dot{\underline{\rm LSF\!}\,}}\kern-.1em\mathop{}
\nolimits^{\rm ind}}
\def\oLSF{\mathop{\bar{\rm LSF}}\nolimits}
\def\oLSFa{\mathop{\bar{\rm LSF}}\nolimits_{\rm al}}
\def\oLSFai{\mathop{\bar{\rm LSF}}\nolimits_{\rm al}^{\rm ind}}
\def\uoLSF{\mathop{\bar{\underline{\rm LSF\!}\,}}\nolimits}
\def\uoLSFa{\mathop{\bar{\underline{\rm LSF\!}\,}}\nolimits_{\rm al}}
\def\uoLSFi{\mathop{\bar{\underline{\rm LSF\!}\,}}\nolimits_{\rm
al}^{\rm ind}}
\def\Ext{\mathop{\rm Ext}\nolimits}
\def\id{\mathop{\rm id}\nolimits}
\def\ha{{\ts\frac{1}{2}}}
\def\Obj{\mathop{\rm Obj\kern .1em}\nolimits}
\def\fObj{\mathop{\mathfrak{Obj}\kern .05em}\nolimits}
\def\fExact{\mathop{\mathfrak{Exact}\kern .05em}\nolimits}
\def\modA{\text{\rm mod-$A$}}
\def\modKQ{\text{\rm mod-$\K Q$}}
\def\modKQI{\text{\rm mod-$\K Q/I$}}
\def\nilKQ{\text{\rm nil-$\K Q$}}
\def\nilKQI{\text{\rm nil-$\K Q/I$}}
\def\bs{\boldsymbol}
\def\ge{\geqslant}
\def\le{\mathord{\leqslant}}
\def\pr{{\mathop{\preceq}\nolimits}}
\def\npr{{\mathop{\npreceq}\nolimits}}
\def\tl{\trianglelefteq\nobreak}
\def\ps{\precsim\nobreak}
\def\ls{{\mathop{\lesssim\kern .05em}\nolimits}}
\def\bF{{\mathbin{\mathbb F}}}
\def\N{{\mathbin{\mathbb N}}}
\def\Z{{\mathbin{\mathbb Z}}}
\def\Q{{\mathbin{\mathbb Q}}}
\def\C{{\mathbin{\mathbb C}}}
\def\K{{\mathbin{\mathbb K\kern .05em}}}
\def\KP{{\mathbin{\mathbb{KP}}}}
\def\cC{{\mathbin{\mathcal C}}}
\def\A{{\mathbin{\mathcal A}}}
\def\B{{\mathbin{\mathcal B}}}
\def\F{{\mathbin{\mathcal F}}}
\def\G{{\mathbin{\mathcal G}}}
\def\H{{\mathbin{\mathcal H}}}
\def\L{{\mathbin{\mathcal L}}}
\def\M{{\mathcal M}}
\def\cP{{\mathbin{\mathcal P}}}
\def\cQ{{\mathbin{\mathcal Q}}}
\def\cR{{\mathbin{\mathcal R}}}
\def\g{{\mathbin{\mathfrak g}}}
\def\h{{\mathbin{\mathfrak h}}}
\def\n{{\mathbin{\mathfrak n}}}
\def\fD{{\mathbin{\mathfrak D}}}
\def\fE{{\mathbin{\mathfrak E}}}
\def\fF{{\mathbin{\mathfrak F}}}
\def\fG{{\mathbin{\mathfrak G}}}
\def\fH{{\mathbin{\mathfrak H}}}
\def\fM{{\mathbin{\mathfrak M}}}
\def\fR{{\mathbin{\mathfrak R}}}
\def\fS{{\mathbin{\mathfrak S}}}
\def\fT{{\mathbin{\mathfrak T}}}
\def\fU{{\mathbin{\mathfrak U\kern .05em}}}
\def\fV{{\mathbin{\mathfrak V}}}
\def\sIp{{\smash{\sst(I,\pr)}}}
\def\sIt{{\smash{\sst(I,\tl)}}}
\def\sJp{{\smash{\sst(J,\pr)}}}
\def\sKt{{\smash{\sst(K,\tl)}}}
\def\sIl{{\smash{\sst(I,\ls)}}}
\def\sJl{{\smash{\sst(J,\ls)}}}
\def\al{\alpha}
\def\be{\beta}
\def\ga{\gamma}
\def\de{\delta}
\def\bde{\bar\delta}
\def\io{\iota}
\def\ep{\epsilon}
\def\la{\lambda}
\def\ka{\kappa}
\def\th{\theta}
\def\ze{\zeta}
\def\si{\sigma}
\def\om{\omega}
\def\De{\Delta}
\def\La{\Lambda}
\def\Om{\Omega}
\def\Ga{\Gamma}
\def\Si{\Sigma}
\def\Th{\Theta}
\def\Up{\Upsilon}
\def\ts{\textstyle}
\def\sst{\scriptscriptstyle}
\def\sm{\setminus}
\def\bu{\bullet}
\def\op{\oplus}
\def\ot{\otimes}
\def\bigop{\bigoplus}
\def\bigot{\bigotimes}
\def\iy{\infty}
\def\ra{\rightarrow}
\def\ab{\allowbreak}
\def\longra{\longrightarrow}
\def\dashra{\dashrightarrow}
\def\hookra{\hookrightarrow}
\def\lt{\ltimes}
\def\el{{\mathbin{\ell\kern .08em}}}
\def\t{\times}
\def\ci{\circ}
\def\ti{\tilde}
\def\md#1{\vert #1 \vert}
\title{Realizing Enveloping Algebras via Moduli Stacks}
\author{Liqian Bai, Fan Xu}
\date{}
\maketitle

\begin{abstract}
Let $\CF(\fObj_\A)$ denote the vector space of $\Q$-valued constructible functions on a given stack $\fObj_\A$ for an exact category $\A$. By using the Ringel--Hall algebra approach, Joyce proved that $\CF(\fObj_\A)$ is an associative $\mathbb{Q}$-algebra via the convolution multiplication and the subspace $\CFi(\fObj_\A)$ of constructible functions supported on indecomposables is a Lie subalgebra of $\CF(\fObj_\A)$ in \cite{joyceA2}. In this paper, we show that there is a subalgebra $\CF^{\text{KS}}(\fObj_\A)$ of $\CF(\fObj_\A)$ isomorphic to the universal enveloping algebra of $\CFi(\fObj_\A)$. Moreover we construct a comultiplication on $\CF^{\text{KS}}(\fObj_\A)$ and a degenerate form of Green's theorem. This generalizes Joyce's work, as well as results of \cite{dxx10}.
\end{abstract}

\section{Introduction}
\label{ai1}
Let $\Lambda$ be a finite dimensional $\mathbb{C}$-algebra such that it is a representative-finite algebra, i.e., there are finitely many finite dimensional indecomposable $\Lambda$-modules up to isomorphism. Let $\mathcal{I}(\Lambda)=\{X_1,\ldots,X_n\}$ be a set of representatives. Let $\mathcal{P}(\Lambda)$ be a set of representatives for the all isomorphism classes of $\Lambda$-modules. There is a free $\mathbb{Z}$-module $R(\Lambda)$ with a basis $\{u_X~|~X\in\mathcal{P}(\Lambda)\}$. Using the Euler characteristic, $\mathcal{P}(\Lambda)$ can be endowed with a multiplicative structure (see \cite{schofield} and \cite{lusztigkyoto90}). The multiplication is defined by
$$
u_X\cdot u_Y=\sum\limits_{A\in\mathcal{P}(\Lambda)}\chi(V(X,Y;A))u_A,
$$
where $V(X,Y;A)=\{0\subseteq A_1\subseteq A~|~A_1\cong X, A/A_1\cong Y\}$ and $\chi(V(X,Y;A))$ is the Euler characteristic of $V(X,Y;A)$. Thus $(R(\Lambda),+,\cdot)$ is a $\mathbb{Z}$-algebra with identity $u_0$. Let $L(\Lambda)$ be a submodule of $R(\Lambda)$ which is spanned by $\{u_X~|~X\in\mathcal{I}(\Lambda)\}$. It follows that $L(\Lambda)$ is a Lie subalgebra of $R(\Lambda)$ with the Lie bracket $[u_X,u_Y]=u_X\cdot u_Y-u_Y\cdot u_X$. Riedtmann studied the universal enveloping algebra of $L(\Lambda)$. Let $R(\Lambda)^{\prime}$ be the subalgebra of $R(\Lambda)$ generated by $\{u_X~|~X\in\mathcal{I}(\Lambda)\}$. Riedtmann showed that $R(\Lambda)^{\prime}$ is isomorphic to the universal enveloping algebra of $L(\Lambda)$. These results have been generalized by two ways.

Joyce generalized Riedtmann's work in the context of constructible functions (also stack functions) over moduli stacks. In \cite{joyceJLMS06}, Joyce defined the Euler characteristics of constructible sets in $\mathbb{K}$-stacks, pushforwards and pullbacks for constructible functions, where $\mathbb{K}$ is an algebraically closed field. Let $\A$ be an abelian category and $\CF(\fObj_\A)$ the vector space of $\mathbb{Q}$-valued constructible functions on $\fObj_\A(\mathbb{K})$, where $\fObj_\A$ is the moduli stack of objects in $\A$ and $\fObj_\A(\mathbb{K})$ the collection of isomorphism classes of objects in $\A$. Joyce proved that $\CF(\fObj_\A)$ is an associative $\mathbb{Q}$-algebra. The algebra $\CF(\fObj_\A)$ can be viewed as a variant of the Ringel-Hall algebra. Let $\CFi(\fObj_\A)$ be a subspace of $\CF(\fObj_\A)$ satisfying the condition that $f([X])\neq0$ implies $X$ is an indecomposable object in $\A$ for every $f\in\CFi(\fObj_\A)$. Then $\CFi(\fObj_\A)$ is shown to be a Lie subalgebra of $\CF(\fObj_\A)$ (\cite[Theorem 4.9]{joyceA2}). Let $\CF_{\rm fin}(\fObj_\A)$ be the subspace of $\CF(\fObj_\A)$ such that
$$
\text{supp}(f)=\big\{[X]\in\fObj_\A(\mathbb{K})~|~f([X])\neq0\big\}
$$
is a finite set for every $f\in\CF_{\rm fin}(\fObj_\A)$. Let
$$
\CFi_{\rm fin}(\fObj_\A)=\CF_{\rm fin}(\fObj_\A)\cap\CFi(\fObj_\A).
$$
Assume that a conflation $X\rightarrow Y\rightarrow Z$ in $\A$ implies that the number of isomorphism classes of $Y$ is finite for all $X,Z\in\Obj(A)$. With the assumption, Joyce proved that $\CF_{\rm fin}(\fObj_\A)$ is an associative algebra and $\CFi_{\rm fin}(\fObj_\A)$ a Lie subalgebra of $\CF_{\rm fin}(\fObj_\A)$. It follows that $\CF_{\rm fin}(\fObj_\A)$ is isomorphic to the universal enveloping algebra of $\CFi_{\rm fin}(\fObj_\A)$. Joyce defined a comultiplication on $\CF_{\rm fin}(\fObj_\A)$ and proved that $\CF_{\rm fin}(\fObj_\A)$ is a bialgebra.

In \cite{dxx10}, the authors extended Riedtmann's results to algebras of representation-infinite type, i.e., the cardinality of isomorphism classes of indecomposable finite dimensional $\Lambda$-modules is infinite. Let $R(\Lambda)$ be the $\mathbb{Z}$-module spanned by $1_{\mathcal{O}}$, where $1_{\mathcal{O}}$ is the characteristic function over a constructible set of stratified Krull-Schmidt $\mathcal{O}$ (see \cite[Section 3]{dxx10}). The subspace $L(\Lambda)$ of $R(\Lambda)$ is spanned by $1_{\mathcal{O}}$, where $\mathcal{O}$ are indecomposable constructible sets. The multiplication is defined by
$$
1_{\mathcal{O}_1}\cdot1_{\mathcal{O}_2}(X)=\chi(V(\mathcal{O}_1,\mathcal{O}_2;X))
$$
where $X$ is a $\Lambda$-module. Then $R(\Lambda)$ is an associative algebra with identity $1_0$ and $L(\Lambda)$ a Lie subalgebra of $R(\Lambda)$ with Lie bracket. The algebra $R(\Lambda)\otimes\mathbb{Q}$ is the universal enveloping algebra of $L(\Lambda)\otimes\mathbb{Q}$. The authors gave the degenerate form of Green's formula and established the comultiplication of $R(\Lambda)$ in \cite{dxx10}.

The goal of this paper is to explicitly construct the enveloping algebra of $\CFi(\fObj_\A)$ by the methods in \cite{dxx10}. Let $\A$ be an exact category satisfying some properties. Let $X\xrightarrow{f}Y\xrightarrow{g}Z$ be a conflation in $\A$ and $\Aut(X\xrightarrow{f}Y\xrightarrow{g}Z)$ the automorphism group of $X\xrightarrow{f}Y\xrightarrow{g}Z$. The key idea in \cite{dxx10} is that $V(X,Y;L)$ has the same Euler characteristic as its fixed point set under the action of $\mathbb{C}^*$. In this paper, we consider exact categories instead of categories of modules. Then as a substitute of the action of $\mathbb{C}^*$, we analyze the action of a maximal torus of $\Aut(X\xrightarrow{f}Y\xrightarrow{g}Z)$ on $X\xrightarrow{f}Y\xrightarrow{g}Z$. The universal enveloping algebra of $\CFi(\fObj_\A)$ can be endowed with a comultiplication structure (Definition \ref{def4.1}). It is compatible with multiplication (Theorem \ref{thm4.6}). The compatibility can be viewed as the degenerate form of Green's theorem on Ringel-Hall algebras (see \cite{green} or \cite{ringel96}).

The paper is organized as follows. In Section 2 we recall the basic concepts about stacks, constructible sets and constructible functions. In Section 3 we define the constructible sets of  stratified Krull-Schmidt. We study the the subspace $\CF^{\text{KS}}(\fObj_\A)$ of $\CF(\fObj_\A)$ generated by characteristic functions supported on constructible sets of stratified Krull-Schmidt. Then $\CF^{\text{KS}}(\fObj_\A)$ provides a realization of the universal enveloping algebra of $\CFi(\fObj_\A)$. In Section $4$ we give the comultiplication $\Delta$ in $\CF^{\rm KS}(\fObj_\A)$ and prove that $\Delta$ is an algebra homomorphism.

\section{Preliminaries}
\label{ai2}

\subsection{Constructible sets and constructible functions}
From now on, let $\mathbb{K}$ be an algebraically closed field with characteristic zero. We recall the definitions of constructible sets and constructible functions on $\mathbb{K}$-stacks. These definitions are taken from Joyce \cite{joyceJLMS06}.
\begin{dfn}
Let $\mathcal{F}$ be a $\mathbb{K}$-stack. Let $\mathcal{F}(\mathbb{K})$ denote the set of $2$-isomorphism classes $[x]$ where $x:\Spec\mathbb{K}\rightarrow\mathcal{F}$ are $1$-morphisms. Every element of $\mathcal{F}(\mathbb{K})$ is called a geometric point (or $\mathbb{K}$-point) of $\mathcal{F}$. For $\mathbb{K}$-stacks $\mathcal{F}$ and $\mathcal{G}$, let $\phi: \mathcal{F}\rightarrow\mathcal{G}$ be a 1-morphism of $\mathbb{K}$-stacks. Then $\phi$ induces a map $\phi_{*}:\mathcal{F}(\mathbb{K})\rightarrow\mathcal{G}(\mathbb{K})$ by $[x]\mapsto [\phi\circ x]$.

For any $[x]\in\mathcal{F}(\mathbb{K})$, let $\Iso_{\mathbb{K}}(x)$ denote the group of 2-isomorphisms $x\rightarrow x$ which is called a stabilizer group. For ease of notations, $\Iso_{\mathbb{K}}(x)$ is used to denote the group instead of $\Iso_{\mathbb{K}}([x])$. If $\Iso_{\mathbb{K}}(x)$ is an affine algebraic $\mathbb{K}$-group for each $[x]\in\mathcal{F}(\mathbb{K})$, then we say $\mathcal{F}$ with affine geometric stabilizers. A morphism of algebraic $\mathbb{K}$-groups $\phi_x:\Iso_{\mathbb{K}}(x)\rightarrow \Iso_{\mathbb{K}}(\phi_{*}(x))$ is induced by $\phi: \mathcal{F}\rightarrow\mathcal{G}$ for each $[x]\in\mathcal{F}(\mathbb{K})$.

A subset $\mathcal{O}\subseteq\mathcal{F}(\mathbb{K})$ is called a constructible set if $\mathcal{O}=\amalg_{i=1}^{n} \mathcal{F}_{i}(\mathbb{K})$ for some $n\in\mathbb{N}^+$, where every $\mathcal{F}_{i}$ is a finite type algebraic $\mathbb{K}$-substack of $\mathcal{F}$. A subset $S\subseteq\mathcal{F}(\mathbb{K})$ is called a locally constructible set if $S\cap\mathcal{O}$ are constructible for all constructible subsets $\mathcal{O}\subseteq\mathcal{F}(\mathbb{K})$. If $\mathcal{O}_1$ and $\mathcal{O}_2$ are constructible sets, then $\mathcal{O}_1\cup\mathcal{O}_2$, $\mathcal{O}_1\cap\mathcal{O}_2$ and $\mathcal{O}_1\setminus\mathcal{O}_2$ are constructible sets by \cite[Lemma 2.4]{joyceJLMS06}.

Let $\Phi:\mathcal{F}(\mathbb{K})\rightarrow\mathcal{G}(\mathbb{K})$ be a map. The set $\Gamma_{\Phi}=\{(x,\Phi(x))~|~x\in\mathcal{F}(\mathbb{K})\}$ is called the graph of $\Phi$. Recall that $\Phi$ is a pseudomorphism if $\Gamma_{\Phi}\bigcap(\mathcal{O}\times\mathcal{G}(\mathbb{K}))$ are constructible for all constructible subsets $\mathcal{O}\subseteq\mathcal{F}(\mathbb{K})$. By \cite[Proposition~4.6]{joyceJLMS06}, if $\phi:\mathcal{F}\rightarrow\mathcal{G}$ is a 1-morphism then $\phi_*$ is a pseudomorphism, $\Phi(\mathcal{O})$ and $\Phi^{-1}(y)\cap\mathcal{O}$ are constructible sets for all constructible subset $\mathcal{O}\subseteq\mathcal{F}(\mathbb{K})$ and $y\in\mathcal{G}(\mathbb{K})$. If $\Phi$ is a bijection and $\Phi^{-1}$ is also a pseudomorphism, we call $\Phi$ a pseudoisomorphism.
\end{dfn}

Then we will recall the definition of the na\"{\i}ve Euler characteristic of a constructible subset of $\mathcal{F}(\mathbb{K})$ in \cite{joyceJLMS06}.

This is a useful result due to Rosenlicht \cite{rosenlicht63}.
\begin{thm}\label{rosenlicht}
Let $G$ be an algebraic $\mathbb{K}$-group acting on a $\mathbb{K}$-variety $X$. There exist an open dense $G$-invariant subset $X_{1}\subseteq X$ and a $\mathbb{K}$-variety $Y$ such that there is a morphism of varieties $\phi:X_{1}\rightarrow Y$ which induces a bijection form $X_{1}(\mathbb{K})/G$ to $Y(\mathbb{K})$.
\end{thm}

Let $X$ be a separated $\mathbb{K}$-scheme of finite type, the Euler characteristic $\chi(X)$ of $X$ is defined by
$$
\chi(X)=\sum\limits_{i=0}^{2\dim X}(-1)^{i}\dim_{\mathbb{Q}_{p}}H_{\text{cs}}^{i}(X,\mathbb{Q}_{p}),
$$
where $p$ is a prime number, $\mathbb{Z}_p=\lim\limits_{\longleftarrow}$ $\mathbb{Z}/p^r\mathbb{Z}$ is the ring of $p$-adic integers, $\mathbb{Q}_p$ is its field of fractions and $H_{\text{cs}}^{i}(X,\mathbb{Q}_{p})$ are the compactly-supported $p$-adic cohomology groups of $X$ for $i\geq0$.

The following properties of Euler characteristic follow \cite{dxx10} and \cite{joyceJLMS06}.
\begin{prop}
\label{prop1}
Let $X$, $Y$ be separated, finite type $\mathbb{K}$-schemes and $\varphi:X\rightarrow Y$ a morphism of schemes. Then:

(1) If $Z$ is a closed subscheme of $X$, then $\chi(X)=\chi(X\setminus Z)+\chi(Z)$.

(2) $\chi(X\times Y)=\chi(X)\times\chi(Y)$.

(3) Let $X$ be a disjoint union of finitely many subschemes $X_{1},\ldots,X_{n}$, we have
$$
\chi(X)=\sum\limits_{i=1}^{n}\chi(X_{i}).
$$

(4) If $\varphi$ is a locally trivial fibration with fibre $F$, then $\chi(X)=\chi(F)\cdot\chi(Y)$.

(5) $\chi(\mathbb{K}^{n})=1$, $\chi(\mathbb{K}\mathbb{P}^{n})=n+1$ for all $n\geq0$.
\end{prop}

An algebraic $\mathbb{K}$-stack $\mathcal{F}$ is said to be stratified by global quotient stacks if $\mathcal{F}(\mathbb{K})=\amalg_{i=1}^s\mathcal{F}_{i}(\mathbb{K})$ for finitely many locally closed substacks $\mathcal{F}_{i}$ where each $\mathcal{F}_{i}$ is 1-isomorphic to a quotient stack $[X_i/G_i]$, where $X_i$ is an algebraic $\mathbb{K}$-variety and $G_i$ a smooth connected linear algebraic $\mathbb{K}$-group acting on $X_i$. By \cite[Propostion 3.5.9]{kresch99}, if $\mathcal{F}$ is a finite type algebraic $\mathbb{K}$-stack with affine geometric stabilizers, then $\mathcal{F}$ is stratified by global quotient stacks.

Let $\mathcal{F}=\amalg_{i=1}^s\mathcal{F}_{i}(\mathbb{K})$ where each $\mathcal{F}_{i}\cong[X_i/G_i]$ as above. By Theorem \ref{rosenlicht}, there exists an open dense $G_i$-invariant subvariety $X_{i1}$ of $X_i$ for each $i$ such that there exists a morphism of varieties $\phi_{i1}:X_{i1}\rightarrow Y_{i1}$, which induces a bijection between $X_{i1}(\mathbb{K})/G_i$ and $Y_{i1}(\mathbb{K})$. Then $\phi_{i1}$ induces a 1-morphism $\theta_{i1}:\mathcal{G}_{i1}\rightarrow Y_{i1}$, where $\mathcal{G}_{i1}$ is 1-isomorphic to $[X_{i1}/G_i]$. Note that $$\text{dim}(X_{i(j-1)}\setminus X_{ij})<\dim X_{i(j-1)}$$ for $j=1,\ldots,k_i$. Using Theorem \ref{rosenlicht} again, we get a stratification
$$
\mathcal{F}(\mathbb{K})=\amalg_{i=1}^{s}\amalg_{j=1}^{k_i}\mathcal{G}_{ij}(\mathbb{K})
$$
for $s,k_i\in\mathbb{N}^+$, where $\mathcal{G}_{ij}\cong[X_{ij}/G_i]$ such that $\phi_{ij}:X_{ij}\rightarrow Y_{ij}$ is a morphism of $\mathbb{K}$-varieties and $\theta_{ij}:\mathcal{G}_{ij}\rightarrow Y_{ij}$ a 1-morphism induced by $\phi_{ij}$. Let $$Y=\amalg_{i=1}^{s}\amalg_{j=1}^{k_i}Y_{ij}~\text{and}~
\Theta=\amalg_{i=1}^{s}\amalg_{j=1}^{k_i}(\theta_{ij})_*:\mathcal{F}(\mathbb{K})\rightarrow Y(\mathbb{K}).
$$
Then $Y$ is a a separated $\mathbb{K}$-scheme of finite type and $\Theta$ a pseudoisomorphism (see \cite[Proposition 4.4 and Proposition 4.7]{joyceJLMS06}).

\begin{dfn}
Let $\mathcal{F}$ be an algebraic $\mathbb{K}$-stack with affine geometric stabilizers and $\mathcal{C}\subseteq\mathcal{F}(\mathbb{K})$ a constructible set. Then $\mathcal{C}$ is pseudo\"{E}\"{E}isomorphic to $Y(\mathbb{K})$, where $Y$ is a separated $\mathbb{K}$-scheme of finite type by \cite[Proposition 4.7]{joyceJLMS06}. The na\"{\i}ve Euler characteristic of $\mathcal{C}$ is defined by $\chi^{\na}(\mathcal{C})=\chi(Y)$.
\end{dfn}

The following lemma is a generalization of Proposition \ref{prop1} (4).
\begin{lem}\label{lem2.1}
Let $\mathcal{F}$ and $\mathcal{G}$ be algebraic $\mathbb{K}$-stacks with affine geometric stabilizers. If $\mathcal{C}\subseteq\mathcal{F}(\mathbb{K})$, $\mathcal{D}\subseteq\mathcal{G}(\mathbb{K})$ are constructible sets, and $\Phi:\mathcal{C}\rightarrow\mathcal{D}$ is a surjective pseudomorphism such that all fibers have the same na\"{\i}ve Euler characteristic $\chi$, then $\chi^{\na}(\mathcal{C})=\chi\cdot\chi^{\na}(\mathcal{D})$.
\end{lem}
\begin{proof}
Because $\mathcal{C}$, $\mathcal{D}$ are constructible sets, there exist separated finite type $\mathbb{K}$-schemes $X$, $Y$ such that $\mathcal{C}$, $\mathcal{D}$ are pseudoisomorphic to $X(\mathbb{K})$, $Y(\mathbb{K})$ respectively. Therefore $\chi^{\na}(\mathcal{C})=\chi(X)$, $\chi^{\na}(\mathcal{D})=\chi(Y)$. Then $\Phi$ induces a surjective pseudomorphism between $X(\mathbb{K})$ and $Y(\mathbb{K})$, say $\phi:X(\mathbb{K})\rightarrow Y(\mathbb{K})$. There exist two projective morphisms
$\pi_{1}:\Gamma_{\phi}\rightarrow X(\mathbb{K})$ and $\pi_{2}:\Gamma_{\phi}\rightarrow Y(\mathbb{K})$. Note that $\pi_{1}$ is also a pseudoisomorphism, that is $\chi^{\na}(\Gamma_{\phi})=\chi(X)$, and all fibres of $\pi_{2}$ have the same na\"{\i}ve Euler characteristic $\chi$. Then $\chi^{\na}(\Gamma_{\phi})=\chi\cdot\chi(Y)$. Hence $\chi(X)=\chi\cdot\chi(Y)$. We finish the proof.
\end{proof}

\begin{dfn}
A function $f:\mathcal{F}(\mathbb{K})\rightarrow\mathbb{Q}$ is called a constructible function on $\mathcal{F}(\mathbb{K})$ if the codomain of $f$ is a finite set and $f^{-1}(a)$ is a constructible subset of $\mathcal{F}(\mathbb{K})$ for each $a\in f(\mathcal{F}(\mathbb{K}))\setminus\{0\}$. Let $\CF(\mathcal{F})$ denote the $\mathbb{Q}$-vector space of all $\mathbb{Q}$-valued constructible functions on $\mathcal{F}(\mathbb{K})$.

Let $S\subseteq\mathcal{F}(\mathbb{K})$ be a locally constructible set. The integral of $f$ on $S$ is
$$
\int_{x\in S}f(x)=\sum\limits_{a\in f(S)\setminus\{0\}}a\chi^{\na}(f^{-1}(a)\cap S)
$$
for each $f\in\CF(\mathcal{F})$.
\end{dfn}

We recall the pushforwards and pullbacks of constructible functions due to Joyce \cite{joyceJLMS06}.
\begin{dfn}
Let $\mathcal{F}$ and $\mathcal{G}$ be algebraic $\mathbb{K}$-stacks with affine geometric stabilizers and $\phi: \mathcal{F}\rightarrow\mathcal{G}$ a 1-morphism. For each $f\in\CF(\mathcal{F})$, the na\"{\i}ve pushforward $\phi^{\na}_!(f):\mathcal{F}(\mathbb{K})\rightarrow\mathbb{Q}$ of $f$ is
$$
\phi^{\na}_!(f)(t)=\sum\limits_{a\in f(\phi^{-1}_{*}(t))\setminus\{0\}}a\chi^{\na}(f^{-1}(a)\cap\phi^{-1}_{*}(t))
$$
for each $t\in\mathcal{G}(\mathbb{K})$. Then $\phi^{\na}_!(f)$ is a constructible function for each $f\in\CF(\mathcal{F})$ by \cite[Theorem~4.9]{joyceJLMS06}.

Similarly, if $\Phi:\mathcal{F}(\mathbb{K})\rightarrow\mathcal{G}(\mathbb{K})$ is a pseudomorphism, the na\"{\i}ve pushforward $\Phi^{\na}_!(f):\mathcal{F}(\mathbb{K})\rightarrow\mathbb{Q}$ of $f\in\CF(\mathcal{F})$ is defined by
$$
\Phi^{\na}_!(f)(t)=\sum\limits_{a\in f(\Phi^{-1}(t))\setminus\{0\}}a\chi^{\na}(f^{-1}(a)\cap\Phi^{-1}(t))
$$
for $t\in\mathcal{G}(\mathbb{K})$. Recall that $\Phi^{\na}_!(f)\in\CF(\mathcal{G})$ by \cite[Theorem~4.9]{joyceJLMS06}.

If $\phi: \mathcal{F}\rightarrow\mathcal{G}$ is a 1-morphism such that $\chi(\text{Ker}(\phi_x))=1$ for all $x\in \mathcal{F}(\mathbb{K})$, we can define a function $m_{\phi}:\mathcal{F}(\mathbb{K})\rightarrow\mathbb{Q}$ by
$$
m_{\phi}(x)=\chi\Big(\Iso_{\mathbb{K}}\big(\phi_*(x)\big)/\phi_x\big(\Iso_{\mathbb{K}}(x)\big)\Big)
$$
for each $x\in\mathcal{F}(\mathbb{K})$. For each $f\in\CF(\mathcal{F})$, the pushforward $\phi_!(f):\mathcal{G}(\mathbb{K})\rightarrow\mathbb{Q}$ of $f$ is defined by
$$
\phi_!(f)=\phi^{na}_!(f\cdot m_{\phi}),
$$
where $(f\cdot m_{\phi})(x)=f(x)m_{\phi}(x)$ for $x\in\mathcal{F}(\mathbb{K})$. Note that $\phi_!(f)\in\CF(\mathcal{G})$ (see \cite{joyceJLMS06}).

If $\phi$ is a 1-morphism of finite type, then $\phi_{*}^{-1}(\mathcal{D})\subset\mathcal{F}(\mathbb{K})$ is a constructible set for each constructible subset $\mathcal{D}$ of $\mathcal{G}(\mathbb{K})$. Then $g\circ\phi_{*}\in\CF(\mathcal{F})$ for $g\in\CF(\mathcal{G})$. Recall that the pullback $\phi^{*}:\CF(\mathcal{G})\rightarrow\CF(\mathcal{F})$ of $\phi$ is defined by $\phi^{*}(g)=g\circ\phi_{*}$ and it is linear.

\end{dfn}

\subsection{Stacks of objects and conflations in $\A$}

From now on, let $(\A,\mathcal{S})$ be a Krull-Schmidt exact $\mathbb{K}$-category with idempotent complete (see \ref{def a1} and \ref{def a3}). For simplicity, we write $\A$ instead of $(\A,\mathcal{S})$.

The isomorphism classes of $X\in\Obj(\A)$ and conflations $X\xrightarrow{i}Y\xrightarrow{d}Z$ in $\A$ are denoted by $[X]$ and $[X\xrightarrow{i}Y\xrightarrow{d}Z]$ (or $[(X,Y,Z,i,d)]$), respectively. Two conflations $X\xrightarrow{i}Y\xrightarrow{d}Z$ and $A\xrightarrow{f}B\xrightarrow{g}C$ are isomorphic if there exist isomorphisms $a:X\rightarrow A$, $b:Y\rightarrow B$ and $c:Z\rightarrow C$ in $\A$ such that the following diagram is communicative
\begin{equation}
\xymatrix{
  X \ar[d]_{a} \ar[r]^{i} & Y \ar[d]_{b} \ar[r]^{d} & Z \ar[d]^{c} \\
  A \ar[r]^{f} & B \ar[r]^{g} & C   }
\end{equation}
The morphism $(a,b,c)$ is called an isomorphism of conflations in $\A$.

\begin{ass}
Assume that $\text{dim}_{\mathbb{K}}\Hom_{\A}(X,Y)$ and $\text{dim}_{\mathbb{K}}\Ext^1_{\A}(X,Y)$ are finite for all $X,Y\in\text{Obj}(\A)$. Let $K(\A)$ denote the quotient group of the Grothendieck group $K_0(\A)$ such that $\tilde{[X]}=0$ in $K(\A)$ implies that $X$ is a zero object in $\A$, where $\tilde{[X]}$ denotes the image of $X$ in $K(\A)$.
\end{ass}

The following $2$-categories are defined in \cite{joyceA1}.

Let $\Sch_{\mathbb{K}}$ be a $2$-category of $\mathbb{K}$-schemes such that objects are $\mathbb{K}$-schemes, $1$-morphisms morphisms of schemes and $2$-morphisms only the natural transformations $\id_{f}$ for all $1$-morphisms $f$. Let $\text{(exactcat)}$ denote the $2$-category of all exact categories with $1$-morphisms exact functors of exact categories and $2$-morphisms natural transformations between the exact functors. If all morphisms of a category are isomorphisms, then the category is called a groupoid. Let (groupoids) be the $2$-category with objects groupoids, $1$-morphisms functors of groupoids and $2$-morphisms natural transformations (also see \cite[Definition 2.8]{joyceA1}).

In \cite[Section 7.1]{joyceA1}, Joyce defined a stack $\mathcal{F}_{\A}:\Sch_{\mathbb{K}}:\rightarrow\text{(exactcat)}$ associated to the exact category $\A$ (the original definition is for abelian category, it can be extended to exact categories directly), where $\mathcal{F}_{\A}$ is a contravariant 2-functor and satisfies the condition $\mathcal{F}_{\A}(\Spec(\mathbb{K}))=\A$. Applying $\mathcal{F}_{\A}$, he defined two moduli stacks
$$
\fObj_\A,\fExact_\A:\Sch_{\mathbb{K}}\rightarrow\text{(groupoids)}
$$
which are contravatiant $2$-functors (\cite[Definition~7.2]{joyceA1}). The $2$-functor
$$\fObj_\A=F\circ\mathcal{F}_{\A},$$
where $F:\text{(exactcat)}\rightarrow\text{(groupoids)}$ is a forgetful $2$-functor as follows. For an exact category $G$, $F(G)$ is a groupoid such that $\text{Obj}(F(G))=\text{Obj}(G)$ and morphisms are isomorphisms in $G$. For $U\in\Sch_{\mathbb{K}}$, a category $\fExact_\A(U)$ is a groupoid whose objects are conflations in $\mathcal{F}_\A(U)$ and morphisms isomorphisms of conflations in $\mathcal{F}_\A(U)$.

Let $\eta:U\rightarrow V$ and $\theta:V\rightarrow W$ be morphisms of schemes in $\Sch_{\mathbb{K}}$. Obviously, the functors $\fObj_\A(\eta):\fObj_\A(V) \rightarrow\fObj_\A(U)$ and $\fExact_\A(\eta):\fExact_\A(V)\rightarrow\fExact_\A(U)$ are induced by $\mathcal{F}_\A(\eta):\mathcal{F}_\A(V)\rightarrow\mathcal{F}_\A(U)$. The natural transformations $\epsilon_{\theta,\eta}:\fObj_{\A}(\eta)\circ\fObj_{\A}(\theta)\rightarrow \fObj_{\A}(\theta\circ\eta)$ and $\epsilon_{\theta,\eta}:\fExact_{\A}(\eta)\circ\fExact_{\A}(\theta)\rightarrow \fExact_{\A}(\theta\circ\eta)$ are also induced by $\epsilon_{\theta,\eta}:\mathcal{F}_{\A}(\eta)\circ\mathcal{F}_{\A}(\theta)\rightarrow \mathcal{F}_{\A}(\theta\circ\eta)$.

Let
$$
K^{\prime}(\A)=\{\tilde{[X]}\in K(\A)~|~X\in\text{Obj}(\A)\}\subset K(\A).
$$

For each $\alpha\in K^{\prime}(\A)$, Joyce defined $\fObj_{\A}^{\alpha}:\Sch_{\mathbb{K}}\rightarrow\text{(groupoids)}$ which is a substack of $\fObj_\A$ in \cite[Definition 7.4]{joyceA1}. For each $U\in\Sch_{\mathbb{K}}$, $\fObj_{\A}^{\alpha}(U)$ is a full subcategory of $\fObj_{\A}(U)$. For each object $X$ in $\fObj_{\A}^{\alpha}(U)$, the image of $\fObj_{\A}(f)(X)$ in $K(\A)$ is $\alpha$ for each morphism $f:\Spec(\mathbb{K})\rightarrow U$.

Let $\eta:U\rightarrow V$ and $\theta:V\rightarrow W$ be morphisms in $\Sch_{\mathbb{K}}$. The functor $\fObj_{\A}^{\alpha}(\eta): \fObj_{\A}^{\alpha}(V)\rightarrow\fObj_{\A}^{\alpha}(U)$ is defined by restriction from $\fObj_{\A}(\eta): \fObj_{\A}(V)\rightarrow\fObj_{\A}(U)$. The natural transformation $\epsilon_{\theta,\eta}:\fObj_{\A}^{\alpha}(\eta)\circ\fObj_{\A}^{\alpha}(\theta)\rightarrow \fObj_{\A}^{\alpha}(\theta\circ\eta)$ is restricted from $\epsilon_{\theta,\eta}:\fObj_{\A}(\eta)\circ\fObj_{\A}(\theta)\rightarrow \fObj_{\A}(\theta\circ\eta)$.

For $\alpha,\beta,\gamma\in K^{\prime}(\A)$ and $\beta=\alpha+\gamma$, $\fExact_{\A}^{\alpha,\beta,\gamma}:\Sch_{\mathbb{K}}\rightarrow \text{(groupoids)}$ is defined as follows. For $U\in\Sch_{\mathbb{K}}$, $\fExact_{\A}^{\alpha,\beta,\gamma}(U)$ is a full subcategory of $\fExact_{\A}(U)$. The objects of $\fExact_{\A}^{\alpha,\beta,\gamma}(U)$ are conflations
$$
X\xrightarrow{i}Y\xrightarrow{d}Z \in\text{Obj}(\fExact_{\A}(U)),
$$
where $X\in\text{Obj}(\fObj_{\A}^{\alpha}(U))$, $Y\in\text{Obj}(\fObj_{\A}^{\beta}(U))$ and $Z\in\text{Obj}(\fObj_{\A}^{\gamma}(U))$. Similarly, the morphism $\fExact_{\A}^{\alpha,\beta,\gamma}(\eta)$ and natural transformation $\epsilon_{\theta,\eta}$ are defined by restriction.

The following theorem is taking from \cite[Theorem 7.5]{joyceA1}.
\begin{thm}
The $2$-functors $\fObj_\A$, $\fExact_\A$ are $\mathbb{K}$-stacks, and $\fObj_{\A}^{\alpha}$, $\fExact_{\A}^{\alpha,\beta,\gamma}$ are open and closed $\mathbb{K}$-substacks of them respectively. There are disjoint unions
$$\fObj_\A=\amalg_{\alpha\in K^{\prime}(\A)}\fObj_\A^{\alpha},\fExact_{\A}=\amalg_{\alpha,\beta,\gamma\in K^{\prime}(\A)\atop \beta=\alpha+\gamma}\fExact_{\A}^{\alpha,\beta,\gamma}.$$
\end{thm}

Assume that $\fObj_\A$ and $\fExact_\A$ are locally of finite type algebraic $\mathbb{K}$-stacks with affine algebraic stabilizers. Recall that $\fObj_\A(\mathbb{K})$ and $\fExact_\A(\mathbb{K})$ are the collection of isomorphism classes of objects in $\A$ and the collection of isomorphism classes of conflations in $\A$, respectively. For each $\alpha\in K^{\prime}(\A)$, $\fObj_\A^{\alpha}(\mathbb{K})$ is the collection of isomorphism classes of $X\in\text{Obj}(\A)$ such that $\tilde{[X]}=\alpha$ (see \cite[Section 3.2]{joyceA2}).

\begin{ex}\label{ex1}
Let $Q=(Q_0,Q_1,s,t)$ be a finite connected quiver, where $Q_0=\{1,\ldots,n\}$ is the set of vertices, $Q_1$ is the set of arrows and $s:Q_1\rightarrow Q_0$ (resp. $t:Q_1\rightarrow Q_0$) is a map such that $s(\rho)$ (resp. $t(\rho)$) is the source (resp. target) of $\rho$ for $\rho\in Q_1$. Let $A=\mathbb{C}Q$ be the path algebra of $Q$ and mod-$A$ denote the category of all finite dimensional right $A$-modules.

Let $\underline{d}=(d_j)_{j\in Q_0}$ for all $d_j\in\mathbb{N}$. There is an affine variety
$$
\text{Rep}(Q,\underline{d})=\bigoplus\limits_{\rho\in Q_1}\Hom(\mathbb{C}^{d_{s(\rho)}},\mathbb{C}^{d_{t(\rho)}}).
$$
For each $x=(x_\rho)_{\rho\in Q_1}\in\text{Rep}(Q,\underline{d})$, there is a $\mathbb{C}$-linear representation $M(x)=(\mathbb{C}^{d_j},x_\rho)_{j\in Q_0,\rho\in Q_1}$ of $Q$. Let $\text{rep}(Q)$ denote the category of finite dimensional $\mathbb{C}$-linear representations of $Q$. Recall that $\text{rep}(Q)\cong \text{mod-}\A$. We identify $\text{rep}(Q)$ with mod-$\A$. The linear algebraic group
$$
\text{GL}(\underline{d})=\prod\limits_{j\in Q_0}\text{GL}(d_j,\mathbb{C})
$$
acts on $\text{Rep}(Q,\underline{d})$ by $g.x=(g_{t(\rho)}x_\rho g^{-1}_{s(\rho)})_{\rho\in Q_1}$ for $g=(g_j)_{j\in Q_0}\in\text{GL}(\underline{d})$.

A complex $M^\bullet=(M^{(i)},\partial^{i})$, where $M^{(i)}\in\text{Obj}(\text{mod-}\A)$ and $\partial^{i+1}\partial^{i}=0$, is bounded if there exist some positive integers $n_0$ and $n_1$ such that $M^{(i)}=0$ for $i\leq -n_0$ or $i\geq n_1$. Let $\underline{\dim}M^{(i)}=\underline{d}^{(i)}$ be the dimension vector of $M^{(i)}$ for each $i\in\mathbb{Z}$. The vector sequence $(\underline{d}^{(i)})_{i\in\mathbb{Z}}$ of $M^{\bullet}$ is denoted by $\underline{\textbf{ds}}(M^{\bullet})$.

Let $\mathcal{C}(Q,\mathbf{\underline{d}})$ denote the affine variety consisting of all complexes $M^\bullet$ with $\underline{\textbf{ds}}(M^{\bullet})= \mathbf{\underline{d}}$. The group $G(\mathbf{\underline{d}})= \prod\limits_{i\in\mathbb{Z}}\text{GL}(\underline{d}^{(i)})$ is a linear algebraic group acting on $\mathcal{C}^{b} (Q,\mathbf{\underline{d}})$. The action is induced by the actions of $\text{GL}(\underline{d}^{(i)})$ on $\text{Rep}(Q,\underline{d}^{(i)})$ for all $i\in\mathbb{Z}$, that is
$$
(g^{(i)})_i.(x^{(i)},\partial^{i})_i=(g^{(i)}.x^{(i)},g^{(i+1)}\partial^{i}(g^{(i)})^{-1})_i.
$$

Let $\{P_1,\ldots,P_n\}$ be a set of representatives for all isomorphism classes of finite dimensional indecomposable projective $A$-modules. A complex $P^\bullet=$
$$
\ldots\rightarrow P^{(i-1)}\xrightarrow{\partial^{i-1}}P^{(i)}\xrightarrow{\partial^{i}}P^{(i+1)}\rightarrow\ldots
$$
is projective if $P^{(i)}\cong\bigoplus\limits_{j=1}^n m^{(i)}_{j}P_{j}$ for $m^{(i)}_{j}\in\mathbb{N}$ and $i\in\mathbb{Z}$. Let
$$
\underline{e}(P^{(i)})=\underline{m}^{(i)}=(m_{1}^{(i)}, \ldots,m_{n}^{(i)})
$$
be a vector corresponding to $P^{(i)}$. By the Krull-Schmidt Theorem, $\underline{e}(P^{(i)})$ is unique. The dimension vector of $P^\bullet$ can be defined by
$$
\mathbf{\underline{dim}}(P^\bullet)=(\ldots,\underline{m}^{(i-1)},\underline{m}^{(i)}, \underline{m}^{(i+1)},\ldots).
$$
A dimension vector $\mathbf{\underline{dim}}(P^\bullet)$ is bounded if $P^\bullet$ is bounded.

let $\mathbf{\underline{m}}=(\underline{m}^{(i)})_{i\in\mathbb{Z}}$ be a bounded dimension vector and $\mathbf{\underline{d}}(\mathbf{\underline{m}})=(\underline{d}^{(i)})_{i\in\mathbb{Z}}$ be the vector sequence of a complex whose dimension vector is $\mathbf{\underline{m}}$. Let $\mathcal{P}^{b}(Q,\mathbf{\underline{m}})$ be the set of all bounded project complexes $P^\bullet$ with $\mathbf{\underline{dim}}(P^\bullet) =\mathbf{\underline{m}}$ and $\underline{\textbf{ds}}(P^{\bullet})= \mathbf{\underline{d}}(\mathbf{\underline{m}})$. Note that $\mathcal{P}^{b}(Q,\mathbf{\underline{m}})$ is a locally closed subset of $\mathcal{C}^{b}(Q,\mathbf{\underline{d}}(\mathbf{\underline{m}}))$. An action of $G(\mathbf{\underline{d}}(\mathbf{\underline{m}}))$ on the variety $\mathcal{P}^{b}(Q,\mathbf{\underline{m}})$ is induced by the action of $G(\mathbf{\underline{d}}(\mathbf{\underline{m}}))$ on $\mathcal{C}^{b}(Q,\mathbf{\underline{d}}(\mathbf{\underline{m}}))$.

Let $\mathcal{P}^b(Q)$ denote the exact category with objects bounded project complexes and morphisms $\phi:P^\bullet\rightarrow Q^\bullet$ morphisms between bounded projective complexes. The Grothendieck group
$$
K_0(\mathcal{P}^b(Q))\cong\bigoplus\limits_{i\in\mathbb{Z}}\mathbb{Z}^{n}_{(i)},
$$
where $\mathbb{Z}^{n}_{(i)}=\mathbb{Z}^{n}$. Note that $K(\mathcal{P}^b(Q))=K_0(\mathcal{P}^b(Q))$ and
$$
K^{\prime}(\mathcal{P}^b(Q))\cong\bigoplus\limits_{i\in\mathbb{Z}}\mathbb{N}^{n}_{(i)},
$$
where $\mathbb{N}^{n}_{(i)}=\mathbb{N}^{n}$.

Joyce defined $\mathcal{F}_{\text{mod}-\mathbb{K}Q}$ in \cite[Example 10.5]{joyceA1}. Similarly, for each $U\in\Sch_{\mathbb{K}}$, we define $\mathcal{F}_{\mathcal{P}^b(Q)}(U)$ to be the category as follows.

The objects of $\mathcal{F}_{\mathcal{P}^b(Q)}(U)$ are complexes of sheaves $P^{\bullet}=(P^{(i)},\partial^i)_{i\in\mathbb{Z}}$, where $P^{(i)}=(\bigoplus_{j\in Q_0}X^{(i)}_{j},x^i)$ and $\partial^{i+1}\partial^{i}=0$. The data $X^{(i)}_{j}$ are locally free sheaves of finite rank on $U$ and $x^i=(x_{\rho}^i)_{\rho\in Q_1}$, where $x_{\rho}^i:X^{(i)}_{s(\rho)}\rightarrow X^{(i)}_{t(\rho)}$ are morphisms of sheaves, such that $P^{(i)}=(\bigoplus_{j\in Q_0}X^{(i)}_{j},x^i)$ are projective $\mathbb{C}Q$-modules for all $i\in\mathbb{Z}$. The morphisms of $\mathcal{F}_{\mathcal{P}^b(Q)}(U)$ are morphisms of complexes $\phi^{\bullet}:(P^{(i)},\partial^i)\rightarrow (Q^{(i)},d^i)$, where $Q^{(i)}=(\bigoplus_{j\in Q_0}Y^{(i)}_{j},y^i)$ and $\phi^{\bullet}$ is a sequence of morphisms
$$(\phi^{i}:P^{(i)} \rightarrow Q^{(i)})_{i\in\mathbb{Z}}$$
with $\phi^i=(\phi^{i}_{j}:X^{(i)}_{j} \rightarrow Y^{(i)}_{j})_{j\in Q_0}$ such that $\phi^{i+1}\partial^i=d^{i}\phi^i$ and $\phi^{i}_{t(\rho)}x_{\rho}^i=y_{\rho}^i\phi^{i}_{s(\rho)}$ for all $i\in\mathbb{Z}$ and $\rho\in Q_1$. It is easy to see that $\mathcal{F}_{\mathcal{P}^b(Q)}(U)$ is an exact category.

Let $\eta:U\rightarrow V$ be a morphism in $\Sch_{\mathbb{K}}$. A functor
$$
\mathcal{F}_{\mathcal{P}^b(Q)}(\eta):\mathcal{F}_{\mathcal{P}^b(Q)}(V)\rightarrow \mathcal{F}_{\mathcal{P}^b(Q)}(U)
$$
is defined as follows. If $(P^{(i)},\partial^i)_{i\in\mathbb{Z}}\in\text{Obj}(\mathcal{F}_{\mathcal{P}^b(Q)}(V))$,
$$
\mathcal{F}_{\mathcal{P}^b(Q)}(\eta)(P^{(i)},\partial^i)_{i\in\mathbb{Z}}=(\eta^*(P^{(i)}),\eta^*(\partial^i))_{i\in\mathbb{Z}}
$$
for $\eta^*(P^{(i)})=\big(\bigoplus_{j\in Q_0}\eta^{*}(X^{(i)}_{j}),(\eta^{*}(x_{\rho}^i))_{\rho\in Q_1}\big)$, where $\eta^{*}(X^{(i)}_{j})$ are the inverse images of $X^{(i)}_{j}$ by the morphism $\eta$, $\eta^*(\partial^i):\eta^*(P^{(i)}) \rightarrow\eta^*(P^{(i+1)})$ with $\eta^*(\partial^{i+1})\eta^*(\partial^i)=0$ for $i\in\mathbb{Z}$ and
$$
\eta^{*}(x_{\rho}^i):\eta^{*}(X^{(i)}_{s(\rho)})\rightarrow \eta^{*}(X^{(i)}_{t(\rho)})
$$
for $\rho\in Q_1$ are pullbacks of morphisms between inverse images. For a morphism $\phi^{\bullet}:(P^{(i)},\partial^i)\rightarrow (Q^{(i)},d^i)$ in $\mathcal{F}_{\mathcal{P}^b(Q)}(V)$, the morphism
$$
\mathcal{F}_{\mathcal{P}^b(Q)}(\eta)(\phi^{\bullet}):\big(\eta^{*}(P^\bullet), \eta^{*}(\partial^i)\big) \rightarrow \big(\eta^{*}(Q^\bullet),\eta^{*}(d^i)\big)
$$
is a sequence of morphisms
$$
 \Big(\eta^{*}(\phi^{i}):\big(\bigoplus_{j\in Q_0}\eta^{*}(X^{(i)}_{j}),(\eta^{*}(x_{\rho}^i))_{\rho}\big)\rightarrow \big(\bigoplus_{j\in Q_0}\eta^{*}(Y^{(i)}_{j}),(\eta^{*}(y_{\rho}^i))_{\rho}\big)\Big)_{i\in\mathbb{Z}},
$$
with $\eta^{*}(\phi^{i+1})\eta^{*}(\partial^i)=\eta^{*} (d^i)\eta^{*}(\phi^i)$, where $\eta^{*}(d^i)$ are pullbacks of morphisms between inverse images which satisfy $\eta^{*}(d^{i+1})\eta^{*}(d^i)=0$, and
$$
\eta^{*}(Q^\bullet)=\Big(\bigoplus_{j\in Q_0}\eta^{*}(Y^{(i)}_{j}),(\eta^{*}(y_{\rho}^i))_{\rho\in Q_1}\Big)_{i\in\mathbb{Z}}
$$
such that the pullbacks
$$
\eta^*(\phi^{i}_{j}):\eta^{*}(X^{(i)}_{j})\rightarrow\eta^{*}(Y^{(i)}_{j})
$$
satisfy $\eta^*(\phi^{i}_{t(\rho)})\eta^*(x_{\rho}^i)=\eta^*(y_{\rho}^i)\eta^*(\phi^{i}_{s(\rho)})$. Because locally free sheaves are flat, $\mathcal{F}_{\mathcal{P}^b(Q)}(\eta)(\phi^{\bullet})$ is an exact functor.

Let $\eta:U\rightarrow V$ and $\theta:V\rightarrow W$ be morphisms in $\Sch_{\mathbb{K}}$. As in \cite[Example 9.1]{joyceA1}, for each $P^\bullet\in\text{Obj}(\mathcal{F}_{\mathcal{P}^b(Q)}(W))$, there is a canonical isomorphism $\epsilon_{\theta,\eta}(P^\bullet):\mathcal{F}_{\mathcal{P}^b(Q)}(\eta)\circ\mathcal{F}_{\mathcal{P}^b(Q)}(\theta)(P^\bullet)\rightarrow \mathcal{F}_{\mathcal{P}^b(Q)}(\theta\circ\eta)(P^\bullet)$. We get a 2-isomorphism of functors
$$
\epsilon_{\theta,\eta}:\mathcal{F}_{\mathcal{P}^b(Q)}(\eta)\circ\mathcal{F}_{\mathcal{P}^b(Q)}(\theta)\rightarrow \mathcal{F}_{\mathcal{P}^b(Q)}(\theta\circ\eta)
$$
by the canonical isomorphisms. Thus we have the $2$-functor $\mathcal{F}_{\mathcal{P}^b(Q)}$.

The set $\fObj_{\mathcal{P}^{b}(Q)}(\mathbb{C})$ consists of all isomorphism classes of complexes in $\mathcal{P}^{b}(Q)$.
\end{ex}

As in \cite[Definition 7.7]{joyceA1} and \cite[Section 3.2]{joyceA2}, we have the following $1$-morphisms
$$
\mathbf{\pi}_l: \fExact_\A\rightarrow\fObj_\A
$$
which induces a map $(\mathbf{\pi}_l)_*: \fExact_\A(\mathbb{K})\rightarrow\fObj_\A(\mathbb{K})$ defined by $[X\xrightarrow{i}Y\xrightarrow{d}Z]\mapsto[X]$;

$$
\pi_m: \fExact_\A\rightarrow\fObj_\A
$$
such that the induced map $(\pi_m)_*: \fExact_\A(\mathbb{K})\rightarrow\fObj_\A(\mathbb{K})$ maps $[X\xrightarrow{i}Y\xrightarrow{d}Z]$ to $[Y]$;
$$
\pi_r: \fExact_\A\rightarrow\fObj_\A
$$
inducing the map $(\pi_r)_*: \fExact_\A(\mathbb{K})\rightarrow\fObj_\A(\mathbb{K})$ by $[X\xrightarrow{i}Y\xrightarrow{d}Z]\mapsto[Z]$.

The map $\mathbf{\pi}_{l*}\times\mathbf{\pi}_{r*}:\fExact_\A(\mathbb{K})\rightarrow\fObj_\A(\mathbb{K})\times\fObj_\A(\mathbb{K})$ is defined by $(\mathbf{\pi}_{l*}\times\mathbf{\pi}_{r*})([X\xrightarrow{i}Y\xrightarrow{d}Z])=([X],[Z])$. Note that $(\pi_l\times\pi_r)_*=\mathbf{\pi}_{l*}\times\mathbf{\pi}_{r*}$.

\section{Hall Algebras}
\label{ai3}
\subsection{Constructible sets of stratified Krull-Schmidt}\label{sec3.1}
These definitions are related to \cite{dxx10}.
\begin{dfn}
Let $\mathcal{O}_1$ and $\mathcal{O}_2$ be two constructible subsets of $\fObj_\A(\mathbb{K})$, the direct sum of $\mathcal{O}_1$ and $\mathcal{O}_2$ is
$$
\mathcal{O}_1\oplus\mathcal{O}_2=\big\{[X_1\oplus X_2]~|~[X_1]\in \mathcal{O}_1,[X_2]\in \mathcal{O}_2 ~\text{and}~X_1, X_2\in \Obj(\mathcal{A})\big\}.
$$
Let $n\mathcal{O}$ denote the direct sum of $n$ copies of $\mathcal{O}$ for $n\in\mathbb{N}^+$ and $0\mathcal{O}=\{[0]\}$. Similarly, let $nX$ denote the direct sum of $n$ copies of $X\in\text{Obj}(\A)$. A constructible subset $\mathcal{O}$ of $\fObj_\A(\K)$ is called indecomposable if $X\in\Obj(\mathcal{A})$ is indecomposable and $X\ncong0$ for every $[X]\in \mathcal{O}$.

A constructible set $\mathcal{O}$ is called to be of Krull-Schmidt if
$$
\mathcal{O}= n_1\mathcal{O}_1\oplus n_2\mathcal{O}_2\oplus\ldots\oplus n_k\mathcal{O}_k,
$$
where $\mathcal{O}_i$ are indecomposable constructible sets and $n_i\in\mathbb{N}$ for $i=1,\ldots,k$. If a constructible set $\mathcal{Q}=\amalg_{i=1}^{n}\mathcal{Q}_i$, where $\mathcal{Q}_i$ are constructible sets of Krull-Schmidt for $1\leq i\leq n$, namely $\mathcal{Q}$ is a disjoint union of finitely many constructible sets of Krull-Schmidt, then $\mathcal{Q}$ is said to be a constructible set of stratified Krull-Schmidt.
\end{dfn}

Let $\mathcal{O}_1$ and $\mathcal{O}_2$ be two indecomposable constructible sets. If $\mathcal{O}_1\cap\mathcal{O}_2\neq\emptyset$ and $\mathcal{O}_1\neq\mathcal{O}_2$, we have
$$
\mathcal{O}_1\oplus\mathcal{O}_2=2(\mathcal{O}_1 \cap \mathcal{O}_2) \amalg \Big(\big(\mathcal{O}_1 \setminus (\mathcal{O}_1 \cap \mathcal{O}_2)\big)\oplus \big(\mathcal{O}_2 \setminus (\mathcal{O}_1 \cap \mathcal{O}_2)\big)\Big)
$$
$$
\amalg \Big((\mathcal{O}_1 \cap \mathcal{O}_2) \oplus \big(\mathcal{O}_2 \setminus (\mathcal{O}_1 \cap \mathcal{O}_2)\big)\Big) \amalg \big((\mathcal{O}_1 \setminus (\mathcal{O}_1 \cap \mathcal{O}_2)) \oplus (\mathcal{O}_1 \cap \mathcal{O}_2)\big).
$$
If $\mathcal{Q}=m_{1}\mathcal{O}_{1}\oplus\ldots\oplus m_{l}\mathcal{O}_{l}$ is a constructible set of Krull-Schmidt, we can write $\mathcal{Q}=\amalg_{i=1}^{n}\mathcal{Q}_i$ as a constructible set of stratified Krull-Schmidt, where
$$
\mathcal{Q}_{i}=n_{i1}\mathcal{O}_{i1}\oplus n_{i2}\mathcal{O}_{i2}\oplus\ldots\oplus n_{ik_{i}}\mathcal{O}_{ik_{i}}
$$
for indecomposable constructible sets $\mathcal{O}_{ij}$ which are disjoint each other. Hence we can assume that $\mathcal{O}_{1}, \ldots,\mathcal{O}_{l}$ are disjoint each other.

Let $\CF^{\text{KS}}(\fObj_\A)$ be the subspace of $\CF(\fObj_\A)$ which is spanned by characteristic functions $1_{\mathcal{O}}$ for constructible sets of stratified Krull-Schmidt $\mathcal{O}$, where each $1_{\mathcal{O}}$ satisfies that $1_{\mathcal{O}}([X])=1$ for $[X]\in\mathcal{O}$, and $1_{\mathcal{O}}([X])=0$ otherwise.

\begin{ex}
Let $\mathbb{P}^{1}$ be the projective line over $\mathbb{K}$ and $\text{coh}(\mathbb{P}^{1})$ denote the category of coherent sheaves on $\mathbb{P}^1$.

Let $O(n)$ denote an indecomposable locally free coherent sheaf whose rank and degree are equal to $1$ and $n$ respectively. Let $S_{x}^{[r]}$ be an indecomposable torsion sheaf such that $\rk(S_{x}^{[r]})=0$, $\text{deg}(S_{x}^{[r]})=r$ and the support of $S_{x}^{[r]}$ is $\{x\}$ for $x\in\mathbb{P}^1$. The Grothendieck group $K_0(\text{coh}(\mathbb{P}^{1}))\cong\mathbb{Z}^{2}$. The data $K(\text{coh}(\mathbb{P}^{1}))$ and $\mathcal{F}_{\text{coh}(\mathbb{P}^{1})}$ are defined in \cite[Example 9.1]{joyceA1}. The set of isomorphism classes of indecomposable objects in $\text{coh}(\mathbb{P}^{1})$ is
$$
\{[S_{x}^{[d]}]~|~x\in\mathbb{P}^1,d\in\mathbb{N}\}\cup\{[O(n)]~|~n\in\mathbb{Z}\}.
$$

Recall that a non-trivial subset $U\subset\mathbb{P}^1$ is closed (resp. open) if $U$ is a finite (resp. cofinite) set. Let $\mathcal{O}_d$ be a finite or cofinite subset of $\{[S_{x}^{[d]}]~|~x\in\mathbb{P}\}$ for each $d\in\mathbb{Z}^{+}$ and $\mathcal{O}_{0}$ a finite subset of $\{[O(n)]~|~n\in\mathbb{Z}\}$. Then $\mathcal{O}_d$ and $\mathcal{O}_{0}$ are indecomposable constructible subsets of $\fObj_{\text{coh}(\mathbb{P}^{1})}(\mathbb{K})$. Note that every indecomposable constructible subset of $\fObj_{\text{coh}(\mathbb{P}^{1})}(\mathbb{K})$ is of the form
$$
\mathcal{O}_{0}\amalg\mathcal{O}_{i_1}\amalg\ldots\amalg\mathcal{O}_{i_n}
$$
for $1\leq i_1<\ldots<i_n$. Then the finite direct sum $\oplus(\mathcal{O}_{0}\amalg\mathcal{O}_{i_1}\amalg\ldots\amalg \mathcal{O}_{i_n})$ is a constructible set of Krull-Schmidt. Every constructible set of Krull-Schmidt in $\fObj_{\text{coh}(\mathbb{P}^{1})}(\mathbb{K})$ is of the form. A constructible set of stratified Krull-Schmidt is a disjoint union of finitely many constructible sets of Krull-Schmidt.
\end{ex}

\begin{ex}
In Example \ref{ex1}, $\fObj_{\mathcal{P}^{b}(Q)}^{\mathbf{\underline{m}}}(\mathbb{C})$ is the set of all isomorphism classes of project complexes in $\mathcal{P}^{b}(Q,\mathbf{\underline{m}})$. Note that
$$
\fObj_{\mathcal{P}^{b}(Q)}(\mathbb{C})=\amalg_{\mathbf{\underline{m}}\in K^{\prime}(\mathcal{P}^b(Q))} \fObj_{\mathcal{P}^{b}(Q)}^{\mathbf{\underline{m}}}(\mathbb{C}).
$$

There is a canonical map
$$
p_{\mathbf{\underline{m}}}:\mathcal{P}^{b}(Q,\mathbf{\underline{m}})\rightarrow \fObj_{\mathcal{P}^{b}(Q)}^{\mathbf{\underline{m}}}(\mathbb{C})
$$
which maps $P^\bullet$ to $[P^\bullet]$. A subset $U\subseteq\fObj_{\mathcal{P}^{b}(Q)}^{\mathbf{\underline{m}}}(\mathbb{C})$ is closed (resp. open) if $p_{\mathbf{\underline{m}}}^{-1}(U)$ is closed (resp. open) in $\mathcal{P}^{b}(Q,\mathbf{\underline{m}})$. A subset $V_{\mathbf{\underline{m}}} \subseteq \fObj_{\mathcal{P}^{b}(Q)}^{\mathbf{\underline{m}}}(\mathbb{C})$ is locally closed if it is an intersection of a closed subset and an open subset of $\fObj_{\mathcal{P}^{b}(Q)}^{\mathbf{\underline{m}}}(\mathbb{C})$. A subset $\mathcal{O}\subseteq \fObj_{\mathcal{P}^{b}(Q)}(\mathbb{C})$ is constructible if it is a finite disjoint union of locally closed sets $V_{\mathbf{\underline{m}}}$. Every indecomposable constructible set $\mathcal{O}$ is of the form $\coprod_{\mathbf{\underline{m}}\in S}V_{\mathbf{\underline{m}}}$, where $S$ is a finite set and each complex in $p_{\mathbf{\underline{m}}}^{-1} (V_{\mathbf{\underline{m}}})$ is an indecomposable complex.
\end{ex}

\subsection{Automorphism groups of conflations}
For each $X\in\text{Obj}(\A)$, suppose that $X=n_1X_1\oplus n_2X_2\oplus\ldots\oplus n_tX_t$, where $X_i$ are indecomposable for $i=1,\ldots,t$ and $X_i\ncong X_j$ for $i\neq j$. Then we have
$$
\Aut(X)\cong(1+rad\End(A))\rtimes\sum\limits_{i=1}^{t}\GL(n_{i},\mathbb{K}).
$$
The rank of maximal torus of $\Aut(X)$ is denoted by $\rk\Aut(X)$. Let $n=n_{1}+n_{2}+\ldots+n_{t}$. Thus the number of indecomposable direct summands of $X$ is $n$, which is denoted by $\gamma(X)$. Note that $\gamma(X)=\rk\Aut(X)$. Let
$$
\gamma(\mathcal{O})=\max\{\gamma(X)~|~[X]\in\mathcal{O}\}
$$
for each constructible set $\mathcal{O}$ in $\fObj_\A(\mathbb{K})$.

Let $X\xrightarrow{f}Y\xrightarrow{g}Z$ be a conflation in $\A$ and $\Aut(X\xrightarrow{f}Y\xrightarrow{g}Z)$ denote the group of $(a_1,a_2,a_3)$ for $a_1\in\Aut(X)$, $a_2\in\Aut(Y)$ and $a_3\in\Aut(Z)$ such that the following diagram is commutative
\begin{equation*}
\xymatrix{
  X \ar[d]_{a_1} \ar[r]^{f} & Y \ar[d]_{a_2} \ar[r]^{g} & Z \ar[d]^{a_3} \\
  X \ar[r]^{f} & Y \ar[r]^{g} & Z   }
\end{equation*}

The homomorphism
$$
p_1:\Aut(X\xrightarrow{f}Y\xrightarrow{g}Z)\rightarrow\Aut(Y)
$$
is defined by $(a_1,a_2,a_3)\mapsto a_2$. If $p_1((a_1,a_2,a_3))=p_1((a^{\prime}_1,a_2,a^{\prime}_3))$ then $f(a_1-a^{\prime}_1)=0$ and ($a_3-a^{\prime}_3)g=0$. We have $a_1=a^{\prime}_1$ and $a_3=a^{\prime}_3$ since $f$ is an inflation and $g$ a deflation. Hence $p_1$ is an injective homomorphism of affine algebraic $\mathbb{K}$-groups and
\begin{equation}\label{equ3}
\rk(\Aut(X\xrightarrow{f}Y\xrightarrow{g}Z))=\rk~\text{Im}p_1\leq\rk\Aut(Y)
\end{equation}

Let
$$
p_2:\Aut(X\xrightarrow{f}Y\xrightarrow{g}Z)\rightarrow\Aut(X)\times\Aut(Z)
$$
be a homomorphism given by $(a_1,a_2,a_3)\mapsto (a_1,a_3)$. If $p_2((a_1,a_2,a_3))=p_2((a_1,a_2^{\prime},a_3))$, then $(a_2-a^{\prime}_2)f=0$ and $g(a_2-a^{\prime}_2)$=0, we have
$$
a_2-a^{\prime}_2\in(\Hom(Z,Y)g)\cap(f\Hom(Y,X)).
$$
Observe that $\text{Ker}p_2$ is a linear space. It follows that $\chi(\text{Ker}p_2)=1$ and
\begin{equation}\label{equ4}
\rk~\text{Im}(p_2)\leq\rk\Aut(X)+\rk\Aut(Z).
\end{equation}

Two conflations $X\xrightarrow{i}Y\xrightarrow{d}Z$ and $X^{\prime}\xrightarrow{i^{\prime}}Y\xrightarrow{d^{\prime}}Z^{\prime}$ in $\A$ are said to be equivalent if there exists a commutative diagram
\begin{equation*}
\xymatrix{
  X \ar[d]_{f} \ar[r]^{i} & Y \ar[d]_{1_{Y}} \ar[r]^{d} & Z \ar[d]^{g} \\
  X^{\prime} \ar[r]^{i^{\prime}} & Y \ar[r]^{d^{\prime}} & Z^{\prime}   }
\end{equation*}
where both $f$ and $g$ are isomorphisms. If the two conflations are equivalent, we write $X\xrightarrow{i}Y\xrightarrow{d}Z\sim X^{\prime}\xrightarrow{i^{\prime}}Y\xrightarrow{d^{\prime}}Z^{\prime}$. The equivalence class of $X\xrightarrow{i}Y\xrightarrow{d}Z$ is denoted by $\langle X\xrightarrow{i}Y\xrightarrow{d}Z\rangle$. Define
$$
V(\mathcal{O}_{1},\mathcal{O}_{2};Y)=\big\{\langle X\xrightarrow{i} Y\xrightarrow{d}Z\rangle~|~X\xrightarrow{i} Y\xrightarrow{d}Z\in\mathcal{S}, [X]\in\mathcal{O}_{1},[Z]\in\mathcal{O}_{2}\big\},
$$
where $\mathcal{S}$ is the collection of all conflations of $\A$.

\subsection{Associative algebras and Lie algebras}
\label{ai33}
For $f,g\in\CF(\fObj_\A)$, define $f\cdot g$ by $(f\cdot g)([X],[Y])=f([X])g([Y])$ for $([X],[Y])\in\fObj_\A(\mathbb{K})\times \fObj_\A(\mathbb{K})$. Thus $f\cdot g\in \CF(\fObj_\A\times\fObj_\A)$. The pushforward of $\pi_m$ is well-defined since $p_1$ is injective. The following definition of multiplication is taken from \cite[Definition~4.1]{joyceA2}.
\begin{dfn}
Using the following diagram
\begin{equation*}
\fObj_\A\times\fObj_\A\xleftarrow{\pi_l\times\pi_r}\fExact_\A\xrightarrow{\pi_m}\fObj_\A,
\end{equation*}
we can define the convolution multiplication
\begin{equation*}
\begin{gathered}
\CF(\fObj_\A\times\fObj_\A)\xrightarrow{(\pi_l\times\pi_r)^*}
\CF(\fExact_\A)\xrightarrow{(\pi_m)_!}\CF(\fObj_\A).
\end{gathered}
\end{equation*}

The multiplication $*:\CF(\fObj_\A)\times\CF(\fObj_\A)\rightarrow\CF(\fObj_\A)$ is a bilinear map defined by

\begin{equation*}
f*g=(\pi_m)_![(\pi_l\times\pi_r)^*(f\cdot g)]=(\pi_m)_![\pi_{l}^*(f)\cdot\pi_{r}^*(g)].
\end{equation*}
\end{dfn}

Let $\mathcal{O}_1$ and $\mathcal{O}_2$ be constructible subsets of $\fObj_\A(\mathbb{K})$, the meaning of $1_{\mathcal{O}_1}\ast1_{\mathcal{O}_2}$ can be understood as follows. The function $m_{\mathbf{\pi}_m}:\fExact_\A(\mathbb{K})\rightarrow\mathbb{Q}$, which is defined by
$$
m_{\mathbf{\pi}_m}([X\xrightarrow{f}Y\xrightarrow{g}Z])= \chi\big[\Aut(Y)/p_1\big(\Aut(X\xrightarrow{f} Y\xrightarrow{g}Z)\big)\big],
$$
is a locally constructible function on $\fExact_\A(\mathbb{K})$ by \cite[Proposition~4.16]{joyceJLMS06}, namely $m_{\mathbf{\pi}_m}|_{\mathcal{O}}$ is a constructible function on $\mathcal{O}$ for every constructible subset $\mathcal{O}\subseteq\fExact_\A(\mathbb{K})$.

For each $[Y]\in\fObj_\A(\mathbb{K})$,
$$
1_{\mathcal{O}_1} \ast 1_{\mathcal{O}_2}([Y])=\sum\limits_{c\in\Lambda_{\mathcal{O}_1\mathcal{O}_2}^Y}c\chi^{na}(Q_{c}),
$$
where
$$
\Lambda_{\mathcal{O}_1\mathcal{O}_2}^Y=\{c=m_{\pi_m}([A\xrightarrow{f}Y\xrightarrow{g}B])~|~[A]\in\mathcal{O}_1, [B]\in\mathcal{O}_2\}\setminus\{0\}
$$
is a finite set, and
$$
Q_{c}=\{[ A\xrightarrow{f}Y\xrightarrow{g}B]~|~[A]\in\mathcal{O}_{1},[B]\in\mathcal{O}_{2},m_{\mathbf{\pi}_m}([ A\xrightarrow{f}Y\xrightarrow{g}B])=c\}
$$
are constructible sets for $c\in\Lambda_{\mathcal{O}_1\mathcal{O}_2}^Y$. In fact, the 1-morphism $\pi_l\times\pi_r$ is of finite type by \cite[Theorem~8.4]{joyceA1}. Hence $(\pi_{l*}\times\pi_{r*})^{-1}(\mathcal{O}_1\times\mathcal{O}_2)$ is a constructible subset of $\fExact_\A$. Then
$$
\Lambda_{\mathcal{O}_1\mathcal{O}_2}^Y=m_{\mathbf{\pi}_m}\big[\big((\pi_{l*}\times\pi_{r*})^{-1}(\mathcal{O}_1\times\mathcal{O}_2)\big) \cap \big((\pi_{m*})^{-1}([Y])\big)\big]\setminus\{0\}
$$
is a finite set by \cite[Proposition~4.6]{joyceJLMS06}. Therefore
$$
Q_{c}=m_{\mathbf{\pi}_m}^{-1}(c)\cap[(\pi_{l*}\times\pi_{r*})^{-1}(\mathcal{O}_1\times\mathcal{O}_2)]\cap((\pi_{m*})^{-1}([Y]))
$$
are constructible for all $c\in\Lambda_{\mathcal{O}_1\mathcal{O}_2}^Y$.

For each $([X],[Z])\in\mathcal{O}_{1}\times\mathcal{O}_{2}$, let
$$
\Lambda_{XZ}^Y=\big\{c=m_{\pi_m}([X\xrightarrow{f}Y\xrightarrow{g}Z])~|~[X\xrightarrow{f}Y\xrightarrow{g}Z] \in\fExact_{\A}(\mathbb{K})\big\}
$$
and
$$
Q_{c}^{X,Z}=\big\{[X\xrightarrow{f}Y\xrightarrow{g}Z]~\Big|~m_{\mathbf{\pi}_m}([X\xrightarrow{f}Y\xrightarrow{g}Z])=c\big\},
$$
where $\Lambda_{XZ}^Y$ is a finite set and $Q_{c}^{X,Z}$ are constructible sets for all $c\in\Lambda_{XZ}^Y$. Then
$$(1_{[X]} \ast 1_{[Z]})([Y])=\sum\limits_{c\in\Lambda_{XZ}^Y}c\chi^{na}(Q_{c}^{X,Z}).$$

Let
$$
\pi_{1}:V(\mathcal{O}_{1},\mathcal{O}_{2};Y)\rightarrow\bigcup_{c\in\Lambda_{\mathcal{O}_1\mathcal{O}_2}^Y}Q_{c}
$$
be a morphism given by $\langle X\xrightarrow{f} Y\xrightarrow{g}Z\rangle\mapsto([X\xrightarrow{f} Y\xrightarrow{g}Z])$. For each fibre of $\pi_{1}$, $\chi^{\na}(\pi_{1}^{-1}([X\xrightarrow{f} Y\xrightarrow{g}Z]))=\chi\Big(\Aut(Y)/p_1\big(\Aut(X\xrightarrow{f} Y\xrightarrow{g}Z)\big)\Big)$. The set
$$
\Big\{\chi\Big(\Aut(Y)/p_1\big(\Aut(X\xrightarrow{f} Y\xrightarrow{g}Z)\big)\Big)~|~[X\xrightarrow{f} Y\xrightarrow{g}Z]\in\bigcup_{c \in\Lambda_{\mathcal{O}_1\mathcal{O}_2}^Y} Q_{c}\Big\}
$$
is finite since $\chi(\Aut(Y)/\text{Im}p_1)=m_{\mathbf{\pi}_m}([X\xrightarrow{f} Y\xrightarrow{g}Z])$.

If $U\subseteq V(\mathcal{O}_{1},\mathcal{O}_{2};Y)$ is a constructible set, then
\begin{equation}\label{formula1}
\chi^{\na}(U)=\sum_{c}c\chi^{na}(P_c),
\end{equation}
where $P_c=\big\{[X\xrightarrow{f} Y\xrightarrow{g}Z]~|~\langle X\xrightarrow{f} Y\xrightarrow{g}Z\rangle\in U, m_{\pi_m}([X\xrightarrow{f} Y\xrightarrow{g}Z])=c\big\}$.

Consequently, we have the na\"{\i}ve Euler characteristics of $V([X],[Z];Y)$ and $V(\mathcal{O}_{1},\mathcal{O}_{2};Y)$.

\begin{lem}
Let $X,Y,Z\in\rm{Obj}(\A)$ and $\mathcal{O}_{1},\mathcal{O}_{2}$ be constructible sets. Then
$$
\chi^{\na}(V([X],[Z];Y))=\sum\limits_{c\in\Lambda_{XZ}^Y}c\chi^{na}(Q_{c}^{X,Z})=1_{[X]}*1_{[Z]}([Y]),
$$
$$
\chi^{\na}(V(\mathcal{O}_{1},\mathcal{O}_{2};Y))=\sum\limits_{c\in\Lambda_{\mathcal{O}_1\mathcal{O}_2}^Y}c\chi^{\na}(Q_c) =1_{\mathcal{O}_1} \ast 1_{\mathcal{O}_2}([Y]).
$$
\end{lem}

The following result is due to \cite[Proposition~6]{dxx10} and \cite[Theorem~4.3]{joyceA2}.
\begin{thm}
The $\mathbb{Q}$-space $\CF(\fObj_\A)$ is an associative $\mathbb{Q}$-algebra, with convolution multiplication $*$ and identity $1_{[0]}$, where $1_{[0]}$ is the characteristic function of $[0]\in\fObj_\A(\mathbb{K})$.
\end{thm}
\begin{proof}
The proof of the theorem is quite similar to that in \cite[Theorem~4.3]{joyceA2} and so is omitted.
\end{proof}

Joyce defined $\CFi(\fObj_\A)$ to be the subspace of $\CF(\fObj_\A)$ such that if $f([X])\neq0$ then $X$ is an indecomposable object in $\A$ for every $f\in\CFi(\fObj_\A)$. There is a result of \cite[Theorem~13]{dxx10} and \cite[Theorem~4.9]{joyceA2}.
\begin{thm}
The $\mathbb{Q}$-space $\CFi(\fObj_\A)$ is a Lie algebra under the Lie bracket $[f,g]=f*g-g*f$ for $f,g\in\CFi(\fObj_\A)$.
\end{thm}
\begin{proof}
The proof is the same as the one used in \cite[Theorem~4.9]{joyceA2}.
\end{proof}

\subsection{The algebra $\CF^{\text{KS}}(\fObj_\A)$}
\begin{lem}\label{lem1}
Let $\mathcal{O}_1$ and $\mathcal{O}_2$ be two constructible subsets of $\fObj_\A(\K)$. For any $Y\in \Obj(\A)$, if $1_{\mathcal{O}_1} \ast 1_{\mathcal{O}_2} ([Y])\neq 0$, then there exist $X,Z \in \Obj(\A)$ such that $[X]\in \mathcal{O}_1$, $[Z]\in \mathcal{O}_2$ and $1_{[X]}*1_{[Z]}([Y])\neq0$.
\end{lem}
\begin{proof}
Let $Q_{c}$ and $\Lambda_{\mathcal{O}_1\mathcal{O}_2}^Y$ be as in Section \ref{ai33}. Let
$$
\mathbf{\pi}_{2}:\bigcup_{c\in\Lambda_{\mathcal{O}_1\mathcal{O}_2}^Y}Q_{c}\rightarrow(\mathbf{\pi}_{l*} \times\mathbf{\pi}_{r*})\Big(\bigcup_{c\in\Lambda_{\mathcal{O}_1\mathcal{O}_2}^Y}Q_c\Big)
$$
be a map which maps $[X\xrightarrow{i} Y\xrightarrow{d}Z]$ to $([X],[Z])$ and
$$
m_{m}=m_{\pi_m}|_{\bigcup_cQ_c}.
$$
It is easy to see that $m_{m}$ is a constructible function over $\bigcup_{c\in\Lambda_{\mathcal{O}_1\mathcal{O}_2}^Y}Q_{c}$.

Because $\pi_l\times\pi_r$ is a $1$-morphism, $\pi_2$ is a pseudomorphism by \cite[Proposition~4.6]{joyceJLMS06}. Thus $\pi_2(\bigcup_{c\in\Lambda_{\mathcal{O}_1\mathcal{O}_2}^{Y}}Q_c)$ is constructible and the na\"{\i}ve pushforward $(\pi_{2})_!^{\na}(m_{m})$ of $m_{m}$ to $\pi_2(\bigcup_{c\in\Lambda_{\mathcal{O}_1\mathcal{O}_2}^{Y}}Q_c)$ exists. Note that $(\pi_{2})_!^{\na}(m_{m})$ is a constructible function on $\pi_2(\bigcup_{c\in\Lambda_{\mathcal{O}_1\mathcal{O}_2}^{Y}}Q_c)$. In fact
$$
(\pi_{2})_!^{\na}(m_{m})([X],[Z])=1_{[X]}*1_{[Z]}([Y])
$$
for all $([X],[Z])\in\pi_2(\bigcup_{c\in\Lambda_{\mathcal{O}_1\mathcal{O}_2}^{Y}}Q_c)$. Therefore
$$
\big\{1_{[X]}*1_{[Z]}([Y])~|~([X],[Z])\in\pi_{2}(\bigcup_{c\in\Lambda_{\mathcal{O}_1\mathcal{O}_2}^{Y}}Q_c)\big\}
$$
is a finite set.

Let $\{([X_1],[Z_1]),\ldots,([X_n],[Z_n])\}$ be a complete set of representatives for $([X],[Z])\in\pi_{2}(\bigcup_{c\in\Lambda_{\mathcal{O}_1\mathcal{O}_2}^{Y}} Q_c)$ such that
$$1_{[X_i]}*1_{[Z_i]}([Y])\neq 1_{[X_j]}*1_{[Z_j]}([Y])$$
for $i\neq j$. Set
$$
\mathcal{P}^{X_k,Z_k}=\Big\{([A],[B])\in\pi_{2}\Big(\bigcup_{c\in\Lambda_{\mathcal{O}_1\mathcal{O}_2}^{Y}} Q_c\Big)~|~1_{[A]}*1_{[B]}([Y])=1_{[X_k]}*1_{[Z_k]}([Y])\Big\}
$$
for all $1\leq k\leq n$.
Then $\mathcal{P}^{X_k,Z_k}$ are constructible sets for all $1\leq k\leq n$ since
$$
\mathcal{P}^{X_k,Z_k}=\big((\pi_{2})_!^{\na}(m_{m})\big)^{-1}(c_k),
$$
where $c_k=(\pi_{2})_!^{\na}(m_{m})([X_k],[Z_k])$. Let $\pi=\pi_2\circ\pi_1$ which maps $\langle X\xrightarrow{i} Y\xrightarrow{d}Z\rangle$ to $([X],[Z])$. For each $([X],[Z])\in\pi_{2}(\bigcup_{c\in\Lambda_{\mathcal{O}_1\mathcal{O}_2}^{Y}}Q_c)$,  $$
\chi^{\na}(\pi^{-1}([X],[Z]))=\chi^{\na}(V([X_k],[Z_k];Y))
$$
for some $k$. According to Lemma \ref{lem2.1}, we have
$$
1_{\mathcal{O}_1} \ast 1_{\mathcal{O}_2}([Y])=\chi^{\na}(V(\mathcal{O}_{1},\mathcal{O}_{2};Y))=\sum\limits_{k=1}^{n}\chi^{\na}(V([X_k],[Z_k];Y))\cdot\chi^{\na} (\mathcal{P}^{X_k,Z_k})
$$
$$
=\sum\limits_{k=1}^{n}(1_{[X_k]}*1_{[Z_k]})([Y])\cdot\chi^{\na}(\mathcal{P}^{X_k,Z_k}).
$$
There exists $([X_k],[Z_k])$ for some $k\in\{1,\ldots,n\}$ such that $1_{[X_k]}*1_{[Z_k]}([Y])\neq0$ since $1_{\mathcal{O}_1} \ast 1_{\mathcal{O}_2}([Y])\neq0$.
\end{proof}

Let $\textbf{D}_n(\mathbb{K})$ denote the group of invertible diagonal matrices in $\textbf{GL}(n,\mathbb{K})$.

The following lemma is related to Riedtmann\cite[Lemma 2.2]{rie94}.
\begin{lem}\label{lem2}
Let $X,Y,Z\in\Obj(\A)$ and $X\xrightarrow{f}Y\xrightarrow{g}Z$ be a conflation in $\A$. If $m_{\pi_m}([X\xrightarrow{f}Y\xrightarrow{g}Z])\neq0$, then $\gamma(Y)\leq\gamma(X)+\gamma(Z)$. In particular, $\gamma(Y)=\gamma(X)+\gamma(Z)$ if and only if $Y\cong X\oplus Z$.
\end{lem}
\begin{proof}
Recall that $m_{\pi_m}([X\xrightarrow{f}Y\xrightarrow{g}Z])=\chi(\Aut Y/\text{Im}(p_1))$.

If $\rk\Aut(Y)>\rk~\text{Im}(p_1)$, then the fibre of the action of a maximal torus of $\Aut(Y)$ on $\Aut Y/\text{Im}(p_1)$ is $(\mathbb{K}^*)^k$ for some $k\geq1$, it forces $\chi(\Aut Y/\text{Im}(p_1))=0$. Hence we have $\rk\Aut(Y)=\rk~\text{Im}(p_1)\leq\rk\Aut(X)+\rk\Aut(Z)$.

We prove the second assertion by induction. First of all, suppose that $X\ncong0$ and $Z\ncong0$. If $\rk\Aut(Y)=2$ and $Y=Y_1\oplus Y_2$, then $\rk\Aut(X)=\rk\Aut(Z)=1$ since $X$ and $Z$ are not isomorphic to $0$. For $t\in\mathbb{K}^*\setminus\{1\}$, $\left(
\begin{array}{cc}
t&0\\
0&t^2
\end{array}
\right)\in\Aut(Y)$ and it is an element of a maximal torus $\textbf{D}_2(\mathbb{K})$ of $\Aut(Y)$. A maximal torus of $\text{Im}(p_1)$ is also a maximal torus of $\Aut(Y)$ since $\rk\Aut(Y)=\rk~\text{Im}(p_1)$. Because two maximal tori of a connected linear algebraic group are conjugate, there exists $\alpha\in\Aut(Y)$ such that $\alpha\left(
\begin{array}{cc}
t&0\\
0&t^2
\end{array}
\right)\alpha^{-1}$ lies in a maximal torus of $\text{Im}(p_1)$. Hence there exist $a\in\Aut(X)$ and $b\in\Aut(Z)$ satisfying $(a,\alpha\left(
\begin{array}{cc}
t&0\\
0&t^2
\end{array}
\right)\alpha^{-1},b)\in\Aut(X\xrightarrow{f}Y\xrightarrow{g}Z)$, namely
$$
(a,\left(
\begin{array}{cc}
t&0\\
0&t^2
\end{array}
\right),b)\in\Aut(X\xrightarrow{\alpha^{-1}f}Y\xrightarrow{g\alpha}Z).
$$
Let $f^{\prime}=\alpha^{-1}f$ and $g^{\prime}=g\alpha$. Observe $(t,\left(
\begin{array}{cc}
t&0\\
0&t
\end{array}
\right),t)\in\Aut(X\xrightarrow{f^{\prime}}Y\xrightarrow{g^{\prime}}Z)$. Hence $f^{\prime}(a-t)=\left(
\begin{array}{cc}
0&0\\
0&t^2-t
\end{array}
\right)f^{\prime}$. Let $s=\frac{1}{t^2-t}(a-t)\in\End(X)$ ($t\neq0,1$). Then $f^{\prime}s=\left(
\begin{array}{cc}
0&0\\
0&1
\end{array}
\right)f^{\prime}$. Because $f^{\prime}$ is an inflation and
$$
f^{\prime}s^2=
\left(
\begin{array}{cc}
0&0\\
0&1
\end{array}
\right)f^{\prime}s=\left(
\begin{array}{cc}
0&0\\
0&1
\end{array}
\right)f^{\prime}=f^{\prime}s,
$$
$s^2=s$. The category $\A$ is idempotent completion, consequently $s$ has a kernel and an image such that $X=\text{Ker}s\oplus\text{Im}s$. But $X$ is indecomposable, without loss of generality we can assume $X=\text{Ker}s$. Then $s=0$. Let $f^{\prime}= {f_{1}\choose f_{2}}$ and $g^{\prime}=(g_1,g_2)$. It follows that
$$
\left(
\begin{array}{c}
0\\
0
\end{array}
\right)=f^{\prime}s=\left(
\begin{array}{cc}
0&0\\
0&1
\end{array}
\right)\left(
\begin{array}{c}
f_1\\
f_2
\end{array}
\right)=\left(
\begin{array}{c}
0\\
f_2
\end{array}
\right).
$$
We have $f_2=0$ and $f^{\prime}=\left(
\begin{array}{c}
f_1\\
0
\end{array}
\right)$. The morphism $Y_1\oplus Y_2\xrightarrow{(0,1)}Y_2$ is a deflation by \cite[Lemma 2.7]{buhler10}. Because $(0,1){f_1\choose0}=0$, there exits $h\in\Hom(Z,Y_1)$ such that $(0,1)=h(g_1,g_2)$. We have $hg_1=0$ and $hg_2=1_{Y_2}$. Observe $g_2h\in\End(Z)$ and $(g_2h)(g_2h)=g_2h$, so $g_2h$ has a kernel $k:K\rightarrow Z$ and an image $i:I\rightarrow Z$. Moreover $Z\cong K\oplus I$. It follows that $Z\cong K$ or $Z\cong I$ since $Z$ is indecomposable. If $Z\cong K$ then $g_2h=0$. But $hg_2h=h$, $K=0$. Thus $h$ is an isomorphism and $g_1=0$. We have $Z\cong Y_2$. Similarly $X\cong Y_1$. Hence $X\oplus Z\cong Y_1\oplus Y_2$.

Assume that the assertion is true for $\rk\Aut(Y)=n<N$. When $n=N$, we can assume $\rk\Aut(X)=n_1$ where $0<n_1<N$, then $\rk\Aut(Z)=N-n_1=n_2$. Let $Y=Y^{\prime}\oplus Y_N$ and $Y^{\prime}=Y_1\oplus\ldots\oplus Y_{N-1}$, where $Y_i$ are indecomposable. Observe that $\left(
\begin{array}{cc}
tI_{N-1}&0\\
0&t^2
\end{array}
\right)$ lies in a maximal torus of $\Aut(Y)$ for $t\in\mathbb{K}^*\setminus\{1\}$. There exists $(a,c,b)\in\Aut(X\xrightarrow{f}Y\xrightarrow{g}Z)$ such that $c$ and $\left(
\begin{array}{cc}
tI_{N-1}&0\\
0&t^2
\end{array}
\right)$ are conjugate in $\Aut(Y)$. For simplicity we assume $c=
\left(
\begin{array}{cc}
tI_{N-1}&0\\
0&t^2
\end{array}
\right)$. So we have the following commutative diagram
\begin{equation*}
\xymatrix{
  X \ar[d]_{a} \ar[r]^{f} & Y^{\prime}\oplus Y_N \ar[d]_{c} \ar[r]^{g} & Z \ar[d]^{b} \\
  X \ar[r]^{f} & Y^{\prime}\oplus Y_N \ar[r]^{g} & Z }
\end{equation*}
where $f=(f_1,f_2,\ldots,f_N)^t$ and $g=(g_1,g_2,\ldots,g_N)$.

There is another commutative diagram
\begin{equation*}
\xymatrix{
  X \ar[d]_{tI_{n_1}} \ar[r]^{(f^*,f_N)^t~} & Y^{\prime}\oplus Y_N \ar[d]_{tI_N} \ar[r]^{~~(g^*,g_N)} & Z \ar[d]^{tI_{n_2}} \\
  X \ar[r]^{(f^*,f_N)^t~} & Y^{\prime}\oplus Y_N \ar[r]^{~~(g^*,g_N)} & Z }
\end{equation*}
where $f^*=(f_1,f_2,\ldots,f_{N-1})^T$ and $g^*=(g_1,g_2,\ldots,g_{N-1})$. Then $f=(f^*,f_N)^t$, $g=(g^*,g_N)$ and $f(a-tI_{n_1}) =\left(
\begin{array}{cc}
0I_{N-1}&0\\
0&t^2-t
\end{array}
\right)f$.
Let
$$
s_{N}=\frac{1}{t^2-t}(a-tI_{n_1}).
$$
Then $fs_{N}=\text{diag}\{0,\ldots,0,1\}f$. It follows $f^*s_N=0$, $f_Ns_N=f_N$ and $g_Nf_N=g\left(
\begin{array}{cc}
0I_{N-1}&0\\
0&1
\end{array}
\right)f=gfs_N=0$. Moreover $s_{N}$ is an idempotent, we know that $X=\text{Ker}s_N\oplus\text{Im}s_N$. If $f_N\neq0$ then $\text{Im}s_N$ is not isomorphic to $0$. Similarly we can define $s_1$, $s_2$, $\ldots$, $s_{N-1}\in\End(X)$ with the property that $fs_i=\text{diag}\{0,\ldots,0,1,0,\ldots,0\}f=(0,\ldots,0,f_i,0,\ldots,0)^t$. Hence $s_i$ is idempotent and if $f_i\neq0$ then $\text{Im}s_i$ is not isomorphic to $0$ for each $i$. Note that $s_1+s_2+\ldots+s_N=1_X\in\Aut(X)$, it follows
$$X=\text{Im}s_1\oplus\ldots\oplus\text{Im}s_N.$$
Hence $f_i=0$ for some $i$ since $\rk\Aut(X)<N$. Without loss of generality, we assume $f_N=0$. Let $(0,\ldots,0,1):Y_1\oplus\ldots\oplus Y_N\rightarrow Y_N$, then
$$(0,\ldots,0,1)(f_1,\ldots,f_N)^t =0$$
Hence there exists $h\in\Hom(Z,Y_N)$ such that $h(g_1,\ldots,g_N)=(0,\ldots,0,1)$, namely $hg_1=0,\ldots,hg_{N-1}=0$ and $hg_N=1$. Therefore $Y_N$ is isomorphic to a direct summand of $Z$. Assume that $Z=Z^{\prime}\oplus Y_N$ where $\gamma(Z^{\prime})=\gamma(Z)-1$. The morphism $(1,0):Z^{\prime}\oplus Y_N\rightarrow Z^{\prime}$ is a deflation, so $g^\prime=g^*(1,0):Y^\prime\rightarrow Z^\prime$ is a deflation by Definition \ref{def a1}. Obviously, $(f_1,\ldots,f_{N-1})^t:X\rightarrow Y_1\oplus\ldots\oplus Y_{N-1}$ is a kernel of $g^\prime$. Thus
$$
X\xrightarrow{(f_1,\ldots,f_{N-1})^t}Y_1\oplus\ldots\oplus Y_{N-1}\xrightarrow{g^\prime}Z^{\prime}
$$
is a conflation. By hypothesis, $Y_1\oplus\ldots\oplus Y_{N-1}\cong X\oplus Z^{\prime}$. Hence $Y=Y_1\oplus\ldots\oplus Y_{N}\cong X\oplus Z$. The proof is completed.
\end{proof}

\begin{rem}\label{rem3.10}
If $1_{[X]}*1_{[Z]}([Y])\neq0$, then $\gamma(Y)\leq\gamma(X)+\gamma(Z)$, where the equality holds if and only if $Y\cong X\oplus Z$.
\end{rem}

\begin{lem}\label{rem3.11}
Let $X,Y,Z\in\Obj(\A)$ and $X\xrightarrow{f}Y\xrightarrow{g}Z$ be a conflation in $\A$. If $m_{\pi_m}([X\xrightarrow{f}Y\xrightarrow{g}Z])\neq0$, $\gamma(Y)<\gamma(X)+\gamma(Z)$ and $Y=Y_1\oplus Y_2$, then there exist two conflations $X_1\xrightarrow{f_1}Y_1\xrightarrow{g_1}Z_1$ and $X_2\xrightarrow{f_2}Y_2\xrightarrow{g_2}Z_2$ in $\A$ such that $X\cong X_1\oplus X_2$, $Z\cong Z_1\oplus Z_2$ and $f=\text{diag}\{f_1,f_2\},g=\text{diag}\{g_1,g_2\}$.
\end{lem}
\begin{proof}
Suppose that $\rk\Aut(X)=n_1$, $\rk\Aut(X)=N$ and $\rk\Aut(Z)=n_2$. Then $N<n_1+n_2$. For simplicity, we use the notation as above. Let $Y=Y_1\oplus\ldots\oplus Y_{N}$, $f=(f_1,f_2,\ldots,f_N)^t$, $g=(g_1,g_2,\ldots,g_N)$ and the isomorphisms $(a,c,b),(tI_{n_1},tI_{N},tI_{n_2})\in\Aut(X\xrightarrow{f}Y\xrightarrow{g}Z)$, where $c=\left(
\begin{array}{cc}
tI_{N-1}&0\\
0&t^2
\end{array}
\right)$. Recall that
$$
s_N=\frac{1}{t^2-t}(a-tI_{n_1})\in\End(X)
$$
is an idempotent such that
$$
fs_{N}=(0,\ldots,0,f_N)^t
$$
and $X=\text{Ker}s_N\oplus\text{Im}s_N$. Similarly, there exists an idempotent
$$
r_N=\frac{1}{t-t^2}(b-tI_{n_2})
$$
in $\End(Z)$ such that $r_{N}g=(0,\ldots,0,g_N)$ and $Z=\text{Ker}r_N\oplus\text{Im}r_N$. Without loss of generality, we assume that $f_N\neq0$ and $g_N\neq0$. Because $f_Ns_N=f_N$ and $r_Ng_N=g_N$,
$$
g_Nf_N=r_Ng_Nf_Ns_N=r_N(g_1,\ldots,g_N)(f_1,\ldots,f_N)^ts_N=0.
$$
It is clear that $i:\text{Ker}s_N\hookrightarrow X$ is a kernel of $f_N:X\rightarrow Y_N$. There exists a morphism $f_{N}^\prime:\text{Im}s_N\rightarrow Y_N$ which is an image of $f_N$ since $X=\text{Ker}s_N\oplus\text{Im}s_N$. Similarly we can find a morphism $g_{N}^\prime:Y_N\rightarrow\text{Im}r_N$ which is a coimage of $g_N$ such that $g_N=jg_{N}^\prime$, where $j:\text{Im}(r_N)\hookrightarrow Z$ is an image of $g_N$. It is easy to check that $f_{N}^\prime$ is an inflation, $g_{N}^\prime$ a deflation and $g_{N}^\prime f_{N}^\prime=0$. Let $h:Y_N\rightarrow A$ be a morphism in $\A$ such that $hf_{N}^\prime=0$. The morphism
$$
(0,\ldots,0,h):Y_1\oplus\ldots\oplus Y_N\rightarrow A
$$
satisfies $(0,\ldots,0,h)f=0$. There exists $k\in\Hom_{\A}(Z,A)$ such that
$$
(0,\ldots,0,h)=kg
$$
since $g$ is a cokernel of $f$. It follows that $h=kg_N=kjg_{N}^\prime$. Hence $g_{N}^\prime$ is a cokernel of $f_{N}^\prime$. Therefore $\text{Im}s_N\xrightarrow{f_{N}^\prime}Y_N\xrightarrow{g_{N}^\prime}\text{Im}r_N$ is a conflation. By induction, every indecomposable direct summand of $Y$ is extended by the direct summands of $X$ and $Z$. The proof is finished.
\end{proof}

\begin{lem}\label{lem3}
Let $\mathcal{O}_1$ and $\mathcal{O}_2$ be two indecomposable constructible subsets of $\fObj_\A(\K)$. Let $A\in\Obj(\A)$ and $\gamma(A)\geq2$. If $[A]\notin\mathcal{O}_1\oplus \mathcal{O}_2$, then $1_{\mathcal{O}_{1}}*1_{\mathcal{O}_{2}}([A])=0$.
\end{lem}
\begin{proof}
If $1_{\mathcal{O}_{1}}*1_{\mathcal{O}_{2}}([A])\neq0$, then there exist $X$, $Y\in \Obj(\mathcal{A})$ such that $[X]\in\mathcal{O}_{1}$, $[Y]\in\mathcal{O}_{2}$ and $1_{[X]}*1_{[Y]}(A)\neq0$ by Lemma~\ref{lem1}. It follows that $\gamma(A)=2$ and $A\cong X\oplus Y$ by Lemma~\ref{lem2} (also see \cite[Theorem~4.9]{joyceA2}). This leads to a contradiction.
\end{proof}

\begin{cor}
Let $\mathcal{O}_{1}$ and $\mathcal{O}_{2}$ be indecomposable subsets of $\fObj_{\A}(\mathbb{K})$. If $\mathcal{O}_{1}\cap\mathcal{O}_{2}=\emptyset$, then
$$
1_{\mathcal{O}_{1}}*1_{\mathcal{O}_{2}}=1_{\mathcal{O}_{1}\oplus\mathcal{O}_{2}}+\sum\limits_{i=1}^{m}a_{i}1_{\mathcal{P}_{i}}
$$
where $\mathcal{P}_{i}$ are indecomposable constructible subsets and $a_i=1_{\mathcal{O}_{1}}*1_{\mathcal{O}_{2}}([X])$ for $[X]\in\mathcal{P}_{i}$.
\end{cor}
\begin{proof}
Let $[M]\in\mathcal{O}_{1}$ and $[N]\in\mathcal{O}_{2}$. Then $M$ is not isomorphic to $N$ since $\mathcal{O}_{1}\cap\mathcal{O}_{2}= \emptyset$. Using the fact that
$m_{\pi_{m}}([M\xrightarrow{(1,0)^t}M\oplus N\xrightarrow{(0,1)}N])=1$, we obtain
$$
1_{\mathcal{O}_{1}}*1_{\mathcal{O}_{2}}([M\oplus N])
$$
$$
=m_{\pi_{m}}([M\xrightarrow{(1,0)^t}M\oplus N\xrightarrow{(0,1)}N])\cdot \chi^{\na}([M\xrightarrow{(1,0)^t}M\oplus N\xrightarrow{(0,1)}N])=1.
$$
By Lemma~\ref{lem3}, we know that if $1_{\mathcal{O}_{1}}*1_{\mathcal{O}_{2}}([X]) \neq0$ and $[X]\notin\mathcal{O}_{1}\oplus\mathcal{O}_{2}$, then $X$ is an indecomposable object. Note that
$$
\big(1_{\mathcal{O}_{1}}*1_{\mathcal{O}_{2}}(\fObj_{\A}(\mathbb{K})\setminus\mathcal{O}_{1}\oplus\mathcal{O}_{2})\big) \setminus\{0\}=\{a_1,a_2,\ldots,a_m\}.
$$
Then $\mathcal{P}_{i}= (1_{\mathcal{O}_{1}}*1_{\mathcal{O}_{2}})^{-1}(a_i)\setminus\mathcal{O}_{1}\oplus\mathcal{O}_{2}$ for $1\leq i\leq m$. We complete the proof.
\end{proof}

Using Lemma~\ref{lem2} and Lemma~\ref{rem3.11}, it is easy to see the following corollary:
\begin{cor}\label{cor1}
Let $\mathcal{O}_{1}$ and $\mathcal{O}_{2}$ be two constructible sets. There exist finitely many constructible sets $\mathcal{Q}_1,\mathcal{Q}_2,\ldots,\mathcal{Q}_n$ such that
$$
1_{\mathcal{O}_{1}}*1_{\mathcal{O}_{2}}=\sum\limits_{i=1}^{n}a_i1_{\mathcal{Q}_{i}}
$$
where $\gamma(\mathcal{Q}_{i})\leq \gamma(\mathcal{O}_{1})+ \gamma(\mathcal{O}_{2})$ and $a_i=(1_{\mathcal{O}_{1}}*1_{\mathcal{O}_{2}})([X])$ for any $[X]\in\mathcal{Q}_{i}$.
\end{cor}

For indecomposable constructible sets $\mathcal{O}_{1},\ldots ,\mathcal{O}_{k}$ and $X\in \Obj(\mathcal{A})$, $1_{\mathcal{O}_{1}}*1_{\mathcal{O}_{2}}*\ldots*1_{\mathcal{O}_{k}}([X])\neq0$ implies that $\gamma(X)\leq k$. In particular, $\gamma(X)=k$ means $X=X_{1}\oplus\ldots\oplus X_{k}$ with $[X_{i}]\in \mathcal{O}_{i}$ for $1\leq i\leq k$.

Let $X_{1},\ldots,X_{m}\in \Obj(\A)$ and there be $r$ isomorphic classes, we can assume that $X_{1},\ldots,X_{m_{1}}$ are isomorphic, $X_{m_{1}+1},\ldots,X_{m_{2}}$ are isomorphic, $\ldots$, and $X_{m_{r-1}+1},\ldots,X_{m_{r}}$ are isomorphic, where $m_{1}+\ldots+m_{r}=m$. By \cite{joyceA2}, we have
\begin{equation}\label{equ1}
\Aut(X_{1}\oplus\ldots\oplus X_{m})/\Aut(X_{1})\times\ldots\times\Aut(X_{m})
\cong\mathbb{K}^{l}\times\prod\limits_{i=1}^{r}(\GL(m_{i},\mathbb{K})/(\mathbb{K}^{*})^{m_{i}}),
\end{equation}

\begin{equation}\label{equ2}
\chi(\Aut(X_{1}\oplus X_{2}\oplus\ldots\oplus X_{m})/\Aut(X_{1})\times\ldots\times\Aut(X_{m}))=\prod\limits_{i=1}^{r}m_{i}!.
\end{equation}

\begin{prop}\label{prop2}
Let $\mathcal{O}$ be an indecomposable constructible set. Then
$$
1^{*k}_{\mathcal{O}}=k!1_{k\mathcal{O}}+\sum\limits_{i=1}^tm_{i}1_{\mathcal{P}_{i}}
$$
where $\gamma(\mathcal{P}_{i})<k$ for each $i$ and $m_{i}=1^{*k}_{\mathcal{O}}([X])$ for $[X]\in\mathcal{P}_{i}$.
\end{prop}
\begin{proof}
We prove the proposition by induction. When $k=1$, it is easy to see that the formula is true. If $k=2$, then
$$
1^{*2}_{\mathcal{O}}([X\oplus X])=1_{\mathcal{O}}([X])\cdot1_{\mathcal{O}}([X])\cdot\chi(\Aut(X\oplus X)/\Aut(X)\times\Aut(X))=2
$$
for $[X]\in\mathcal{O}$ and
$$
1^{*2}_{\mathcal{O}}([X\oplus Y])=
$$
$$
\big(1_{\mathcal{O}}([X])\cdot1_{\mathcal{O}}([Y])+1_{\mathcal{O}}([Y])\cdot1_{\mathcal{O}}([X])\big)\cdot\chi\big(\Aut(X\oplus Y)/\Aut(X)\times\Aut(Y)\big)=2,
$$
where $[X],[Y]\in \mathcal{O}$ and $X\ncong Y$. If $[X]\notin\mathcal{O}\oplus\mathcal{O}$ and $\gamma(X)\geq2$ then $1^{*2}_{\mathcal{O}}([X])=0$ by Lemma \ref{lem3}. Hence $1^{*2}_{\mathcal{O}}=2\cdot1_{\mathcal{O}\oplus\mathcal{O}}+\sum\limits_{i}m_{i}\mathcal{P}_{i}$ where $\mathcal{P}_{i}$ are indecomposable constructible sets by Corollary \ref{cor1}.

Now we suppose that the formula is true for $k\leq n$. When $k=n+1$, we have
$$
1^{*(n+1)}_{\mathcal{O}}=1^{*(n)}_{\mathcal{O}}*1_{\mathcal{O}}=(n!1_{n\mathcal{O}}+\sum c_{\mathcal{P}^{\prime}}1_{\mathcal{P}^{\prime}})*1_{\mathcal{O}},
$$
where $\mathcal{P}^{\prime}$ are constructible sets with $\gamma(\mathcal{P}^{\prime})<n$. If the formula is true for $k=n+1$, then
$$
n!1_{n\mathcal{O}}*1_{\mathcal{O}}= (n+1)!1_{(n+1)\mathcal{O}}+\sum c_{\mathcal{Q}}1_{\mathcal{Q}},
$$
where $\mathcal{Q}$ are constructible sets with $\gamma(\mathcal{Q})<n+1$.
Hence it suffices to show that the initial term of $1_{n\mathcal{O}}*1_{\mathcal{O}}$ is $(n+1)1_{(n+1)\mathcal{O}}$, namely $(1_{n\mathcal{O}}*1_{\mathcal{O}})([X])=n+1$ for all $[X]\in(n+1)\mathcal{O}$.

Assume that $X=m_1X_1\oplus m_2X_2\oplus\ldots\oplus m_rX_r$, where $X_1,\ldots,X_r\in\text{Obj}(\A)$ which are not isomorphic to each other, $[X_i]\in\mathcal{O}$ for $1\leq i\leq r$, $m_1,\ldots,m_r$ are positive integers and $m_1+m_2+\ldots+m_r=n+1$.

$$
(1_{n\mathcal{O}}*1_{\mathcal{O}})([X])=(1_{[(m_1-1)X_1\oplus m_2X_2\oplus\ldots\oplus m_rX_r]}*1_{[X_1]})([X])
$$
$$
+(1_{[m_1X_1\oplus (m_2-1)X_2\oplus\ldots\oplus m_rX_r]}*1_{[X_2]})([X])
$$
$$
+\ldots
$$
$$
+(1_{[m_1X_1\oplus\ldots\oplus m_{r-1}X_{r-1}\oplus (m_r-1)X_r]}*1_{[X_r]})([X])
$$
Using Equation (\ref{equ2}), it follows that
$$
1_{[X_1]}^{*(m_1-1)}*1_{[X_2]}^{*m_2}*\ldots*1_{[X_r]}^{*m_r}=(m_1-1)!m_2!\ldots m_r!1_{[(m_1-1)X_1\oplus m_2X_2\oplus\ldots\oplus m_rX_r]}+\ldots,
$$
$$
1_{[X_1]}^{*(m_1-1)}*1_{[X_2]}^{*m_2}*\ldots*1_{[X_r]}^{*m_r}*1_{[X_1]}=(\prod\limits_{i=1}^{r}m_i!)1_{[m_1X_1\oplus m_2X_2\oplus\ldots\oplus m_rX_r]}+\ldots
$$
Compare the initial monomials of the two equations, it follows that
$$
1_{[(m_1-1)X_1\oplus m_2X_2\oplus\ldots\oplus m_rX_r]}*1_{[X_1]}=m_11_{[m_1X_1\oplus m_2X_2\oplus\ldots\oplus m_rX_r]}+\ldots
$$
Thus $1_{[(m_1-1)X_1\oplus m_2X_2\oplus\ldots\oplus m_rX_r]}*1_{[X_1]}([X])=m_1$.

Similarly, we have $1_{[m_1X_1\oplus \ldots\oplus(m_i-1)X_i\oplus\ldots\oplus m_rX_r]}*1_{[X_i]}([X])=m_i$ for $i=2,\ldots,r$. Hence $(1_{n\mathcal{O}}*1_{\mathcal{O}})([X])=\sum\limits_{i=1}^r m_i=n+1$
which completes the proof.
\end{proof}

By induction, we have the following corollary.
\begin{cor}\label{cor2}
Let $\mathcal{O}_{1},\mathcal{O}_{2},\ldots,\mathcal{O}_{k}$ be indecomposable constructible sets which are pairwise disjoint. Then we have the following equations
$$
1^{*n_{1}}_{\mathcal{O}_{1}}*1^{*n_{2}}_{\mathcal{O}_{2}}\ldots*1^{*n_{k}}_{\mathcal{O}_{k}}= n_{1}!n_{2}!\ldots n_{k}!1_{n_{1}\mathcal{O}_{1}\oplus\ldots\oplus n_{k}\mathcal{O}_{k}}+\ldots,
$$
$$
1_{m_{1}\mathcal{O}_{1}\oplus\ldots\oplus m_{k}\mathcal{O}_{k}}*1_{n_{1}\mathcal{O}_{1}\oplus\ldots\oplus n_{k}\mathcal{O}_{k}}= \prod\limits_{i=1}^{k}\frac{(m_{i}+n_{i})!}{m_{i}!n_{i}!}1_{(m_{1}+n_{1})\mathcal{O}_{1}\oplus\ldots\oplus (m_{k}+n_{k})\mathcal{O}_{k}}+\ldots,
$$
where $k$ is a positive integer and $m_1,\ldots,m_k,n_1,\ldots,n_k\in\mathbb{N}$.
\end{cor}

Let $\text{Ind}(\alpha)$ be the subset of $\fObj_\A^{\alpha}(\mathbb{K})$ such that X are indecomposable for all $[X]\in\text{Ind}(\alpha)$.
\begin{lem}\label{lem3.17}
For each $\alpha\in K^{\prime}(\A)$, $\rm Ind(\alpha)$ is a locally constructible set.
\end{lem}
\begin{proof}
Assume $\alpha,\beta,\gamma\in K^{\prime}(\A)\setminus\{0\}$. The map
$$
f:\coprod_{\beta,\gamma;\atop\beta+\gamma=\alpha}\fObj_\A^{\beta}(\mathbb{K})\times\fObj_\A^{\gamma}(\mathbb{K})\rightarrow\fObj_\A^{\alpha}(\mathbb{K})
$$
is defined by $([B],[C])\mapsto[B\oplus C]$. It is clear that $f$ is a pseudomorphism. Every $\fObj_\A^{\beta}(\mathbb{K})\times\fObj_\A^{\gamma}(\mathbb{K})$ is a locally constructible set. For any constructible set $\mathcal{C}\subseteq\fObj_\A(\mathbb{K})\times\fObj_\A(\mathbb{K})$, there are finitely many $\fObj_\A^{\beta}(\mathbb{K})\times\fObj_\A^{\gamma}(\mathbb{K})$ such that $\mathcal{C}\cap (\fObj_\A^{\beta}(\mathbb{K})\times\fObj_\A^{\gamma} (\mathbb{K}))\neq\emptyset$. Hence $\amalg_{\beta,\gamma;\beta+\gamma=\alpha}\fObj_\A^{\beta}(\mathbb{K})\times\fObj_\A^{\gamma}(\mathbb{K})$ is locally constructible. Then $\text{Im}f$ is a locally constructible set. It follows that $\text{Ind}(\alpha)=\fObj_\A^{\alpha}(\mathbb{K})\setminus\text{Im}f$ is locally constructible.
\end{proof}

The following proposition is due to \cite[Proposition~11]{dxx10}.
\begin{prop}\label{prop3}
Let $\mathcal{O}_{1},\mathcal{O}_{2}$ be two constructible sets of Krull-Schmidt. It follows that
$$
1_{\mathcal{O}_{1}}*1_{\mathcal{O}_{2}}=\sum\limits_{i=1}^{c}a_{i} 1_{\mathcal{Q}_{i}}
$$
for some $c\in\mathbb{N}^+$, where $a_{i}=1_{\mathcal{O}_{1}}*1_{\mathcal{O}_{2}}([X])$ for each $[X]\in\mathcal{Q}_{i}$ and $\mathcal{Q}_{i}$ are constructible sets of stratified Krull-Schmidt such that $\gamma(\mathcal{Q}_i)\leq \gamma(\mathcal{O}_1)+\gamma(\mathcal{O}_2)$.
\end{prop}
\begin{proof}
Because $\mathcal{O}_{1},\mathcal{O}_{2}$ are constructible sets, the equation holds for some constructible sets $\mathcal{Q}_{i}$ with $\gamma(\mathcal{Q}_i)\leq \gamma(\mathcal{O}_1)+\gamma(\mathcal{O}_2)$ by Corollary~\ref{cor1}.

For every $[Y_{i}]\in\mathcal{Q}_{i}$, $1_{\mathcal{O}_{1}}*1_{\mathcal{O}_{2}}([Y_i])\neq0$. By Lemma~\ref{lem1}, there exist $X_{i},Z_{i}\in\Obj(\mathcal{A})$ such that $[X_{i}]\in \mathcal{O}_{1}$, $[Z_{i}]\in \mathcal{O}_{2}$ and $1_{[X_{i}]}*1_{[Z_{i}]}([Y_{i}])\neq0$ since $1_{\mathcal{O}_{1}}*1_{\mathcal{O}_{2}}([Y_{i}])\neq0$. Thanks to Lemma~\ref{lem2}, we have that $\gamma(Y_i)\leq\gamma(X_i)+\gamma(Z_i)$. According to Lemma~\ref{rem3.11}, all indecomposable direct summands of $Y_{i}$ are extended by the direct summands of $X_{i}$ and $Z_{i}$ since $1_{[X_{i}]}*1_{[Z_{i}]}([Y_{i}])\neq0$.

By the discussion in Section \ref{sec3.1}, we can suppose that $\mathcal{O}_{1}=\bigoplus\limits_{i=1}^{t}a_{i}\mathcal{C}_{i}$ and $\mathcal{O}_{2}=\bigoplus\limits_{j=1}^{t}b_{j}\mathcal{C}_{j}$, where $a_{i},b_{j}\in\{0,1\}$ for all $i,j$ and $\mathcal{C}_{i}$ are indecomposable constructible sets such that $\mathcal{C}_{i}\cap\mathcal{C}_{j}=\emptyset$ or $\mathcal{C}_{i}=\mathcal{C}_{j}$ for all $i\neq j$. Let $1\leq r\leq t$, the set
$$
\{A_{1},A_{2},\ldots,A_{r}~|~\emptyset\neq A_{i}\subseteq \{1,\ldots,n\}~\text{for}~i=1,\ldots,r\}
$$
is called an $r$-partition of $\{1,2,\ldots,t\}$ if $A_{1}\cup A_{2}\cup\ldots\cup A_{r}=\{1,2,\ldots,t\}$ and $A_{i}\cap A_{j}=\emptyset$ for all $i\neq j$. Obviously, the cardinal number of all partitions of $\{1,2,\ldots,t\}$ is finite. Let $\{A_{1},A_{2},\ldots,A_{r}\}$, $\{B_{1},B_{2},\ldots,B_{r}\}$ be two $r$-partitions of $\{1,2,\ldots,t\}$ and $c_{k}\in\mathbb{Q}\setminus\{0\}$ for $k=1,2,\ldots,r$. Set $\mathcal{O}_{A_{k}}=\bigoplus\limits_{i\in A_{k}}a_{i}\mathcal{C}_{i}$ and $\mathcal{O}_{B_{k}}=\bigoplus\limits_{j\in B_{k}}b_{j}\mathcal{C}_{j}$ for $1\leq k\leq r$. Then we have
$$
\mathcal{R}_{A_{k},B_{k},c_{k}}=\{[X]\in\mathcal{O}_{A_{k}}\oplus \mathcal{O}_{B_{k}}~|~1_{\mathcal{O}_{A_{k}}}*1_{\mathcal{O}_{B_{k}}}([X])=c_{k}\},
$$
$$
\mathcal{I}_{A_{k},B_{k},c_{k}}=\{[X]~|~X~\text{indecomposable},1_{\mathcal{O}_{A_{k}}}*1_{\mathcal{O}_{B_{k}}}([X])=c_{k}\}.
$$
This means that for each $[X]\in\mathcal{R}_{A_{k},B_{k},c_{k}}$, there exist $[A]\in\mathcal{O}_{A_{k}}$ and $[B]\in\mathcal{O}_{B_{k}}$ such that $X\cong A\oplus B$. For each $[Y]\in\mathcal{I}_{A_{k},B_{k},c_{k}}$, there exist $[C]\in\mathcal{O}_{A_{k}}$ and $[D]\in\mathcal{O}_{B_{k}}$ such that $C\rightarrow Y\rightarrow D$ is a non-split conflation in $\A$. Note that
$$
\mathcal{R}_{A_{k},B_{k},c_{k}}=((1_{\mathcal{O}_{A_{k}}}*1_{\mathcal{O}_{B_{k}}})^{-1}(c_{k}))\cap(\mathcal{O}_{A_{k}}\oplus \mathcal{O}_{B_{k}}).
$$
By Corollary~\ref{cor2}, $\mathcal{R}_{A_{k},B_{k},c_{k}}=\emptyset$ or $\mathcal{O}_{A_{k}}\oplus \mathcal{O}_{B_{k}}$. Hence $\mathcal{R}_{A_{k},B_{k},c_{k}}$ is a constructible set of Krull-Schmidt. There exist $\alpha_1,\ldots,\alpha_s\in K^{\prime}(\A)$ such that $\mathcal{I}_{A_{k},B_{k},c_{k}}=(\amalg_{i=1}^s\text{Ind}(\alpha_i))\cap ((1_{\mathcal{O}_{A_{k}}}*1_{\mathcal{O}_{B_{k}}})^{-1}(c_{k}))$. By Lemma~\ref{lem3.17}, $\mathcal{I}_{A_{k},B_{k},c_{k}}$ is an indecomposable constructible set.

Finally, $1_{\mathcal{O}_{1}}*1_{\mathcal{O}_{2}}$ is a $\mathbb{Q}$-linear combination of finitely many  $1_{\oplus_{k=1}^{r}\mathcal{O}_{A_{k},B_{k},c_{k}}}$, where $\mathcal{O}_{A_{k},B_{k},c_{k}}$ run through $\mathcal{R}_{A_{k},B_{k},c_{k}}$ and $\mathcal{I}_{A_{k},B_{k},c_{k}}$ for all $r$-partitions and $r=1,2,\ldots,t$. We finish the proof.
\end{proof}

Thus we summarize what we have proved as the following theorem which is due to \cite[Theorem~12]{dxx10}.
\begin{thm}
The $\mathbb{Q}$-space $\CF^{\rm{KS}}(\fObj_\A)$ is an associative $\mathbb{Q}$-algebra with convolution multiplication $*$ and identity $1_{[0]}$.
\end{thm}

\subsection{The universal enveloping algebra of $\CFi(\fObj_\A)$}
\label{ai42}

Let $U(\CFi(\fObj_\A))$ denote the universal enveloping algebra of $\CFi(\fObj_\A)$ over $\mathbb{Q}$. The multiplication in $U(\CFi(\fObj_\A))$ will be written as $(x,y)\mapsto xy$. There is a $\mathbb{Q}$-algebra homomorphism
$$
\Phi :U(\CFi(\fObj_\A))\rightarrow\CF^{\text{KS}}(\fObj_\A)
$$
defined by $\Phi(1)=1_{[0]}$ and $\Phi(f_1f_2\ldots f_n)=f_1*f_2*\ldots*f_n$, where $f_1,f_2,\ldots,f_n$ belong to $\CFi(\fObj_\A)$.

The following theorem is related to \cite[Theorem 15]{dxx10}.
\begin{thm}
$\Phi :U(\CFi(\fObj_\A))\rightarrow\CF^{\rm KS}(\fObj_\A)$ is an isomorphism.
\end{thm}
\begin{proof}
For simplicity of presentation, let
$$U=U(\CFi(\fObj_\A))~\text{and}~\CF=\CF^{\text{KS}}(\fObj_\A).$$
Suppose that $\mathcal{O}_{1},\mathcal{O}_{2},\ldots,\mathcal{O}_{k}$ are indecomposable constructible subsets of $\fObj_\A(\mathbb{K})$ which are pairwise disjoint. It follows that $1_{\mathcal{O}_{1}},1_{\mathcal{O}_{2}},\ldots, 1_{\mathcal{O}_{k}}$ are linearly independent in $\CFi(\fObj_\A)$.

Let $U_{\mathcal{O}_{1}\ldots\mathcal{O}_{k}}$ denote the subspace of $U$ which is spanned by all $1_{\mathcal{O}_{1}}^{n_1}1_{\mathcal{O}_{2}}^{n_2}\ldots1_{\mathcal{O}_{k}}^{n_k}$ for $n_i\in\mathbb{N}$ and $i=1,\ldots,k$.

Define $\CF_{\mathcal{O}_{1}\ldots\mathcal{O}_{n}}$ to be the subalgebra of $\CF$ which is generated by the elements $1_{n_1\mathcal{O}_{1}\oplus n_2\mathcal{O}_{2}\oplus\ldots\oplus n_k\mathcal{O}_{k}}$ of $\CF$, where $n_i\in\mathbb{N}$ for $i=1,2,\ldots,k$.

The homomorphism $\Phi$ induces a homomorphism
$$
\Phi_{\mathcal{O}_{1}\ldots\mathcal{O}_{k}}:U_{\mathcal{O}_{1}\ldots\mathcal{O}_{k}}\rightarrow \CF_{\mathcal{O}_{1}\ldots\mathcal{O}_{k}}
$$
which maps $1_{\mathcal{O}_{1}}^{n_1} 1_{\mathcal{O}_{2}}^{n_2}\ldots1_{\mathcal{O}_{k}}^{n_k}$ to $1_{\mathcal{O}_{1}}^{*n_1}*1_{\mathcal{O}_{2}}^{*n_2}*\ldots*1_{\mathcal{O}_{k}}^{*n_k}$.

First of all, we want to show that $\Phi_{\mathcal{O}_{1}\ldots\mathcal{O}_{k}}$ is injective.

For $m\in\mathbb{N}$, let $U_{\mathcal{O}_{1}\ldots\mathcal{O}_{k}}^{(m)}$ be the subspace of $U$ which is spanned by
$$
\big\{1_{\mathcal{O}_{1}}^{n_1}1_{\mathcal{O}_{2}}^{n_2}\ldots1_{\mathcal{O}_{k}}^{n_k}~|~\sum\limits_{i=1}^{k}n_i\leq m, n_i\geq0 ~\text{for}~i=1,\ldots,k\big\}
$$
Using the PBW Theorem, we obtain that
$$
\big\{1_{\mathcal{O}_{1}}^{n_1}1_{\mathcal{O}_{2}}^{n_2}\ldots1_{\mathcal{O}_{k}}^{n_k}~|~\sum\limits_{i=1}^{k}n_i= m, n_i\geq0 ~\text{for}~i=1,\ldots,k\big\}
$$
is a basis of the $\mathbb{Q}$-vector space $U_{\mathcal{O}_{1}\ldots\mathcal{O}_{k}}^{(m)}/ U_{\mathcal{O}_{1}\ldots\mathcal{O}_{k}}^{(m-1)}$ for $m\geq1$.

Similarly, we define $\CF_{\mathcal{O}_{1}\ldots\mathcal{O}_{k}}^{(m)}$ to be a subspace of $\CF_{\mathcal{O}_{1}\ldots\mathcal{O}_{k}}$ such that each $f\in\CF_{\mathcal{O}_{1}\ldots\mathcal{O}_{k}}^{(m)}$ is of the form $\sum\limits_{i=1}^l c_{i}1_{\mathcal{C}_i}$, where $l\in\mathbb{N}^+$, $c_{i}\in\mathbb{Q}$, $1_{\mathcal{C}_i}\in \CF_{\mathcal{O}_{1}\ldots\mathcal{O}_{k}}$ and $\mathcal{C}_i$ are constructible sets of Krull-Schmidt such that $\gamma(\mathcal{C}_i)\leq m$.

In $\CF^{(m)}/\CF^{(m-1)}$, the set
$$
\{1_{n_1\mathcal{O}_{1}\oplus n_2\mathcal{O}_{2}\oplus\ldots\oplus n_k\mathcal{O}_{k}}~|~\sum\limits_{i=1}^{k}n_i= m, n_i\geq0 ~\text{for}~i=1,\ldots,k\}
$$
is linearly independent by the Krull-Schmidt Theorem.

For each $m\geq1$, $\Phi_{\mathcal{O}_{1}\ldots\mathcal{O}_{k}}$ induce a map
$$
\Phi_{\mathcal{O}_{1}\ldots\mathcal{O}_{k}}^{(m)}:U_{\mathcal{O}_{1}\ldots\mathcal{O}_{k}}^{(m)}/ U_{\mathcal{O}_{1}\ldots\mathcal{O}_{k}}^{(m-1)}\rightarrow\CF_{\mathcal{O}_{1}\ldots\mathcal{O}_{k}}^{(m)}/ \CF_{\mathcal{O}_{1}\ldots\mathcal{O}_{k}}^{(m-1)}
$$
which maps $1_{\mathcal{O}_{1}}^{n_1}1_{\mathcal{O}_{2}}^{n_2}\ldots1_{\mathcal{O}_{k}}^{n_k}$ to $n_1!n_2!\ldots n_k! 1_{n_1\mathcal{O}_{1}\oplus n_2\mathcal{O}_{2}\oplus\ldots\oplus n_k\mathcal{O}_{k}}$ (also see Corollary~\ref{cor2}), where $\sum\limits_{i=1}^{k}n_i= m$ and $m_i\geq0$. From this we know that $\Phi_{\mathcal{O}_{1}\ldots\mathcal{O}_{k}}^{(m)}$ is injective for all $m\in\mathbb{N}$.
Obviously, both $U_{\mathcal{O}_{1}\mathcal{O}_{2}\ldots\mathcal{O}_{n}}$ and $\CF_{\mathcal{O}_{1}\ldots\mathcal{O}_{n}}$ are filtered. From the properties of filtered algebra, we know that $\Phi_{\mathcal{O}_{1}\ldots\mathcal{O}_{k}}$ is injective. Hence $\Phi :U\rightarrow\CF$ is injective.

Finally, we show that $\Phi$ is surjective by induction. When $m=1$, the statement is trivial. Then we assume that every constructible function $f=\sum\limits_{i=1}^t a_{i}1_{\mathcal{Q}_i}$ lies in $\Im(\Phi)$, where $a_{i}\in\mathbb{Q}$ and $\mathcal{Q}_i$ are constructible sets of stratified Krull-Schmidt with $\gamma(\mathcal{Q}_i)<m$.

Let $n_1+n_2+\ldots+n_k=m$ and $n_i\in\mathbb{N}$ for $1\leq i\leq k$. Then
$$
\Phi(1_{\mathcal{O}_{1}}^{n_1}1_{\mathcal{O}_{2}}^{n_2}\ldots1_{\mathcal{O}_{k}}^{n_k})= 1_{\mathcal{O}_{1}}^{*n_1}*1_{\mathcal{O}_{2}}^{*n_2}*\ldots*1_{\mathcal{O}_{k}}^{*n_k}
$$
$$
=n_{1}!n_{2}!\ldots n_{n}!1_{n_{1}\mathcal{O}_{1}\oplus n_{2}\mathcal{O}_{2}\oplus\ldots\oplus n_{k}\mathcal{O}_{k}}+ \sum\limits_{j=1}^s b_{j}1_{\mathcal{P}_{j}},
$$
where $b_j\in\mathbb{Q}$ and $\mathcal{P}_{j}$ are constructible sets of stratified Krull-Schmidt with $\gamma(\mathcal{P}_{j})<m$. By the hypothesis, $\sum\limits_{j=1}^{s}b_{j}1_{\mathcal{P}_{j}}\in\Im(\Phi)$. Hence $1_{n_{1}\mathcal{O}_{1}\oplus n_{2}\mathcal{O}_{2}\oplus\ldots\oplus n_{k}\mathcal{O}_{k}}$ lies in $\Im(\Phi)$. The algebra $\CF$ is generated by all $1_{n_{1}\mathcal{O}_{1}\oplus\ldots\oplus n_{k}\mathcal{O}_{k}}$, which proves that $\Phi$ is surjective, the proof is finished.
\end{proof}

\section{Comultiplication and Green's formula}

\subsection{Comultiplication}
We now turn to define a comultiplication on the algebra $\CF^{\text{KS}}(\fObj_\A)$. For $f,g\in\CF(\fObj_\A)$, $f\otimes g$ is define by $f\otimes g([X],[Y])=f([X])g([Y])$ for $([X],[Y])\in (\fObj_\A\times\fObj_\A)(\mathbb{K})= \fObj_\A(\mathbb{K})\times\fObj_\A(\mathbb{K})$ (see \cite[Difinition~4.1]{joyceA2}). Let $X\xrightarrow{f}Y\xrightarrow{g} Z$ be a conflation in $\A$. Recall that the map $p_2:\Aut(X\xrightarrow{f}Y\xrightarrow{g}Z)\rightarrow\Aut(X)\times\Aut(Z)$ is defined by $(a_1,a_2,a_3)\mapsto (a_1,a_3)$ and $\text{Ker}p_2=1$. Therefore the pushforward of $\mathbf{\pi}_l\times\mathbf{\pi}_r$ is well-defined.

The following definitions are related to \cite[Section~6]{dxx10} and \cite[Defintion~4.16]{joyceA2}.
\begin{dfn}\label{def4.1}
From now on, assume that $\mathbf{\pi}_m:\fExact_\A\rightarrow\fObj_\A$ is of finite type. Then we have the following diagram
\begin{equation*}
\begin{gathered}
\CF^{\text{KS}}(\fObj_\A\times\fObj_\A)\xleftarrow{(\mathbf{\pi}_l\times\mathbf{\pi}_r)_!}
\CF^{\text{KS}}(\fExact_\A)\xleftarrow{(\mathbf{\pi}_m)^{*}}\CF^{\text{KS}}(\fObj_\A).
\end{gathered}
\end{equation*}
The comultiplication
\begin{equation*}
\Delta:\CF^{\text{KS}}(\fObj_\A)\rightarrow\CF^{\text{KS}}(\fObj_\A\times\fObj_\A)
\end{equation*}
is defined by $\Delta=(\mathbf{\pi}_l\times\mathbf{\pi}_r)_!\circ(\mathbf{\pi}_m)^*$, where $\CF^{\text{KS}}(\fObj_\A\times\fObj_\A)$ is regarded as a topological completion of $\CF^{\text{KS}}(\fObj_\A)\otimes\CF^{\rm KS}(\fObj_\A)$.

The counit $\varepsilon:\CF^{\text{KS}}(\fObj_\A)\rightarrow\mathbb{Q}$ maps $f$ to $f([0])$.
\end{dfn}

Note that $\Delta$ is a $\mathbb{Q}$-linear map since $(\mathbf{\pi}_l\times\mathbf{\pi}_r)_!$ and $(\mathbf{\pi}_m)^*$ are $\mathbb{Q}$-linear map.

\begin{dfn}
Let $\alpha=[A],\beta=[B]\in\fObj_\A(\mathbb{K})$ and $\mathcal{O}\subseteq\fObj_\A(\mathbb{K})$ be a constructible set of stratified Krull-Schmidt, define
$$
h^{\beta\alpha}_{\mathcal{O}}=\Delta(1_{\mathcal{O}})([A],[B]).
$$
Let $\mathcal{O}_1$ and $\mathcal{O}_2\subseteq\fObj_\A(\mathbb{K})$ be constructible sets, define
$$
g^{\alpha}_{\mathcal{O}_2\mathcal{O}_1}=1_{\mathcal{O}_1}*1_{\mathcal{O}_2}(\alpha).
$$
\end{dfn}

Because $\Delta(1_{\mathcal{O}})$ is a constructible function, $\Delta(1_{\mathcal{O}})= \sum\limits_{i=1}^{n}h^{\beta_i\alpha_i}_{\mathcal{O}} 1_{\mathcal{O}_i}$ for some $\alpha_i,\beta_i\in\fObj_\A(\mathbb{K})$ and $n\in \mathbb{N}$, where $\mathcal{O}_i$ are constructible subsets of $\fObj_\A(\mathbb{K})\times\fObj_\A(\mathbb{K})$.

\begin{lem}\label{lem4.2}
Let $X$, $Y$, $Z\in\Obj(\A)$. If $X\oplus Z$ is not isomorphic to $Y$, then $\Delta(1_{[Y]})([X],[Z])=0$.
\end{lem}
\begin{proof}
If $\Delta(1_{[Y]})([X],[Z])\neq0$, there exists a conflation $X\xrightarrow{f}Y\xrightarrow{g}Z$ in $\A$ such that $m_{\pi_l\times\pi_r}([X\xrightarrow{f}Y\xrightarrow{g}Z])\neq0$. Recall that
$$
m_{\pi_l\times\pi_r}([X\xrightarrow{f}Y\xrightarrow{g}Z])=\chi\big((\Aut(X)\times\Aut(Z))/\text{Im}p_2\big).
$$
If $\rk~\text{Im}p_{2}<\rk\big(\Aut(X)\times\Aut(Z)\big)$, the fibre of the action of a maximal torus of $\Aut(X)\times\Aut(Z)$ on $(\Aut(X)\times\Aut(Z))/\text{Im}p_2$ is $(\mathbb{K}^*)^{l}$ for $l>0$. Then $\chi\big((\Aut(X)\times\Aut(Z))/\text{Im}p_2\big)=0$, which is a contradiction. Hence $\rk\big(\Aut(X)\times\Aut(Z)\big)=\rk~\text{Im}p_{2}$.

Assume that $\rk\Aut(X)=n_1$, $\rk\Aut(Z)=n_2$ and $\rk\Aut(Y)=n$ for some positive integers $n_1$, $n_2$ and $n$. Note that $\textbf{D}_{n_1}\times\textbf{D}_{n_2}$ is a maximal torus of $\Aut(X) \times\Aut(Z)$. Because $\rk\big(\Aut(X)\times\Aut(Z)\big)=\rk~\text{Im}(p_{2})$, each maximal torus of $\text{Im}p_2$ is also a maximal torus of $\Aut(X) \times\Aut(Z)$. Therefore every maximal torus of $\text{Im}p_2$ and $\textbf{D}_{n_1}\times\textbf{D}_{n_2}$ are conjugate. For simplicity, we can assume that $\textbf{D}_{n_1}\times\textbf{D}_{n_2}$ is a maximal torus of $\text{Im}p_{2}$. For $(t_1I_{n_1},t_2I_{n_2})\in\textbf{D}_{n_1}\times\textbf{D}_{n_2}$, where $t_1\neq t_2$, there exists $\tau\in\Aut(Y)$ such that $(t_1I_{n_1},\tau,t_2I_{n_2})\in\Aut(X\xrightarrow{f}Y\xrightarrow{g}Z)$. Then we have the commutative diagram
\begin{equation*}
\xymatrix{
  X \ar[d]_{t_1I_{n_1}} \ar[r]^{f} & Y \ar[d]_{\tau} \ar[r]^{g} & Z \ar[d]^{t_2I_{n_2}} \\
  X \ar[r]^{f} & Y \ar[r]^{g} & Z   }
\end{equation*}
The morphism $(t_2I_{n_1},t_2I_{n},t_2I_{n_2})$ is also in $\Aut(X\xrightarrow{f}Y\xrightarrow{g}Z)$. The following diagram is commutative
\begin{equation*}
\xymatrix{
  X \ar[d]_{t_2I_{n_1}} \ar[r]^{f} & Y \ar[d]_{t_2I_{n}} \ar[r]^{g} & Z \ar[d]^{t_2I_{n_2}} \\
  X \ar[r]^{f} & Y \ar[r]^{g} & Z   }
\end{equation*}
Consequently $g(\tau-t_2I_{n})=0$. Because $f$ is a kernel of $g$, there exists $h\in\Hom(Y,X)$ such that $\tau-t_2I_{n}=fh$. Then $\tau=fh+t_2I_{n}$. We have
$$
f(t_1I_{n_1})=\tau f=(fh+t_2I_{n}))f,
$$
it follows that
$$
fhf=f(t_1I_{n_1})-(t_2I_{n})f=f(t_1I_{n_1}-t_2I_{n_1}).
$$
Then $hf=(t_1-t_2)I_{n_1}$ since $f$ is an inflation. Let $f^\prime=\frac{1}{t_1-t_2}h$, then $f^\prime f=1_X$. Hence $X$ is isomorphic to a direct summand of $Y$. The proof is completed.
\end{proof}

For an indecomposable object $X\in\Obj(\A)$, direct summands of $X$ are only $X$ and $0$. Thus $\Delta(1_{[X]})=1_{[X]}\otimes1_{[0]}+ 1_{[0]}\otimes1_{[X]}$. It follows that $\Delta(f)=f\otimes1_{[0]}+1_{[0]}\otimes f$ for $f\in\CFi(\fObj_\A)$.

By Lemma \ref{lem4.2}, $h^{\beta\alpha}_{\mathcal{O}}=1$ if $\alpha\oplus\beta\in\mathcal{O}$, and $h^{\beta\alpha}_{\mathcal{O}}=0$ otherwise. Let $\mathcal{O}=n_1\mathcal{O}_1\oplus\ldots\oplus n_m\mathcal{O}_m$ be a constructible set of stratified Krull-Schmidt, where $\mathcal{O}_i$ are indecomposable constructible sets for all $1\leq i\leq m$. By Lemma \ref{lem4.2}, the formula $\Delta(1_{\mathcal{O}})= \sum\limits_{i=1}^{n}h^{\beta_i\alpha_i}_{\mathcal{O}} 1_{\mathcal{O}_i}$ can be written as
$$
\Delta(1_{\mathcal{O}})=\sum\limits_{1\leq i\leq m;0\leq k_i\leq n_i}1_{k_1\mathcal{O}_1\oplus\ldots\oplus k_m\mathcal{O}_m}\otimes 1_{(n_1-k_1)\mathcal{O}_1\oplus\ldots\oplus (n_m-k_m)\mathcal{O}_m}.
$$

Hence we have the following proposition.
\begin{prop}
Let $\mathcal{O}$ be a constructible set of stratified Krull-Schmidt, then $\Delta(1_{\mathcal{O}})\in\CF^{\rm KS}(\fObj_\A)\otimes\CF^{\rm KS}(\fObj_\A)$, i.e., the map
$$
\Delta:\CF^{\rm KS}(\fObj_\A)\rightarrow\CF^{\rm KS}(\fObj_\A)\otimes\CF^{\rm KS}(\fObj_\A))
$$
is well-defined.
\end{prop}

\subsection{Green's formula on stacks}
Recall that
$$
\int_{x\in S}f(x)=\sum\limits_{a\in f(S)\setminus\{0\}}a\chi^{\na}(f^{-1}(a)\cap S),
$$
where $f$ is a constructible function and $S$ a locally constructible set.

Let $\mathcal{O}_1,\mathcal{O}_2,\mathcal{O}_\rho,\mathcal{O}_\sigma,\mathcal{O}_\epsilon,\mathcal{O}_\tau,\mathcal{O}_{\lambda}$ be constructible sets and $\alpha\in\mathcal{O}_1,\beta\in\mathcal{O}_2,\rho\in\mathcal{O}_\rho, \sigma\in\mathcal{O}_\sigma,\epsilon\in\mathcal{O}_\epsilon, \tau\in\mathcal{O}_\tau, \lambda\in\mathcal{O}_{\lambda}$ such that $\mathcal{O}_\rho\oplus\mathcal{O}_\sigma=\mathcal{O}_1$ and $\mathcal{O}_\epsilon\oplus\mathcal{O}_\tau=\mathcal{O}_2$.

The following theorem is the degenerate form of Green's formula which is related to \cite[Theorem~22]{dxx10}.
\begin{thm}
\label{Greenthm}
Let $\mathcal{O}_1,\mathcal{O}_2$ be constructible subsets of $\fObj_\A(\mathbb{K})$ and $\alpha^\prime,\beta^\prime\in\fObj_\A(\mathbb{K})$, then we have
$$
g^{\alpha^\prime\oplus\beta^\prime}_{\mathcal{O}_2\mathcal{O}_1}=\int_{\rho,\sigma,\epsilon,\tau\in\fObj_\A(\mathbb{K}); \rho\oplus\sigma\in\mathcal{O}_1,\epsilon\oplus\tau\in\mathcal{O}_2} g^{\alpha^\prime}_{\epsilon\rho}g^{\beta^{\prime}}_{\tau\sigma}.
$$
\end{thm}
\begin{proof}
By the proof of Lemma \ref{lem1},  $g^{\alpha^\prime\oplus\beta^\prime}_{\mathcal{O}_2\mathcal{O}_1}=\int_{\alpha\in\mathcal{O}_1,\beta\in\mathcal{O}_2} g^{\alpha^\prime\oplus\beta^\prime}_{\beta\alpha}$. It suffices to prove the following formula
$$
g^{\alpha^\prime\oplus\beta^\prime}_{\beta\alpha}=\int_{\rho,\sigma,\epsilon,\tau\in\fObj_\A(\mathbb{K}); \rho\oplus\sigma=\alpha,\epsilon\oplus\tau=\beta}g^{\alpha^\prime}_{\epsilon\rho}g^{\beta^{\prime}}_{\tau\sigma}.
$$

Suppose that $[A]=\alpha,[B]=\beta,[A^{\prime}]=\alpha^{\prime},[B^{\prime}]=\beta^{\prime},[C]=\rho,[D]=\sigma, [E]=\epsilon$ and $[F]=\tau$ for $A,B,C,D,E,F\in\text{Obj}(\A)$. There are finitely many $(\rho,\sigma)$ and $(\epsilon,\tau)$ such that $\rho\oplus\sigma=\alpha$ and $\epsilon\oplus\tau=\beta$. The morphism
$$
i:\bigcup\limits_{[C],[D],[E],[F];\atop[C\oplus D]=[A],[E\oplus F]=[B]}V([C],[E];A^\prime)\times V([D],[F];B^{\prime})\rightarrow V([A],[B];A^\prime\oplus B^\prime)
$$
is defined by
$$
(\langle C\xrightarrow{f_1}A^\prime\xrightarrow{g_1}E\rangle, \langle D\xrightarrow{f_2}B^\prime\xrightarrow{g_2}F\rangle)\mapsto \langle C\oplus D\xrightarrow{f}A^\prime\oplus B^\prime\xrightarrow{g}E\oplus F\rangle,
$$
where $f=\left(
\begin{array}{cc}
f_1&0\\
0&f_2
\end{array}
\right)$ and $g=\left(
\begin{array}{cc}
g_1&0\\
0&g_2
\end{array}
\right)$. Because both $C\xrightarrow{f_1}A^\prime\xrightarrow{g_1}E$ and $D\xrightarrow{f_2}B^\prime\xrightarrow{g_2}F$ are conflations, $C\oplus D\xrightarrow{f} A^\prime\oplus B^\prime\xrightarrow{g}E\oplus F$ is a conflation by \cite[Proposition 2.9]{buhler10}. Hence the morphism is well-defined. Note that $i$ is injective and $g^{\alpha^\prime}_{\epsilon\rho}g^{\beta^{\prime}}_{\tau\sigma}=\chi^{\na}(V([C],[E];A^\prime)\times V([D],[F];B^{\prime}))$.

By \cite[Lemma 4.2]{joyceJLMS06}, we have
$$
\chi^{\na}(V([A],[B];A^\prime\oplus B^\prime))=\chi^{\na}(\text{Im}i)+\chi^{\na}\big(V([A],[B];A^\prime\oplus B^\prime) \setminus\text{Im}i\big).
$$
According to Lemma \ref{rem3.11}, if $m_{\pi_m}([A\xrightarrow{f}A^\prime\oplus B^\prime\xrightarrow{g}B])\neq0$, then there exist two conflations $C\xrightarrow{f_1}A^\prime\xrightarrow{g_1}E$ and $D\xrightarrow{f_2}B^\prime\xrightarrow{g_2}F$ in $\A$ such that $A\cong C\oplus D$, $B\cong E\oplus F$, $f=\left(
\begin{array}{cc}
f_1&0\\
0&f_2
\end{array}
\right)$ and $g=\left(
\begin{array}{cc}
g_1&0\\
0&g_2
\end{array}
\right)$. Thus
$$
m_{\pi_{m}}([A\xrightarrow{f}A^\prime\oplus B^\prime\xrightarrow{g}B])=0
$$
for any $\langle A\xrightarrow{f}A^\prime\oplus B^\prime\xrightarrow{g}B\rangle \in V([A],[B];A^\prime\oplus B^\prime)\setminus \text{Im}i$. Using (\ref{formula1}), it follows that $\chi^{\na}(V([A],[B];A^\prime\oplus B^\prime)\setminus\text{Im}i)=0$. Hence
\begin{eqnarray*}
g^{\alpha^\prime\oplus\beta^\prime}_{\beta\alpha}=\chi^{\na}(V([A],[B];A^\prime\oplus B^\prime))=\chi^{\na}(\text{Im}i)\\
=\int_{\rho,\sigma,\epsilon,\tau\in\fObj_\A(\mathbb{K}); \rho\oplus\sigma=\alpha,\epsilon\oplus\tau=\beta}g^{\alpha^\prime}_{\epsilon\rho}g^{\beta^{\prime}}_{\tau\sigma}.
\end{eqnarray*}
This completes the proof.
\end{proof}

For all $f_1,f_2,g_1,g_2\in\CF^{\rm KS}(\fObj_\A)$, define $(f_1\otimes g_1)*(f_2\otimes g_2)=(f_1*f_2)\otimes(g_1*g_2)$. Using Green's formula, we have the following theorem due to \cite[Theorem~24]{dxx10}.
\begin{thm}\label{thm4.6}
The map $\Delta:\CF^{\rm KS}(\fObj_\A)\rightarrow\CF^{\rm KS}(\fObj_\A)\otimes\CF^{\rm KS}(\fObj_\A)$ is an algebra homomorphism.
\end{thm}
\begin{proof}
The proof is similar to the one in \cite[Theorem 24]{dxx10}. Let $\mathcal{O}_1, \mathcal{O}_2\in\fObj_\A(\mathbb{K})$ be constructible sets of stratified Krull-Schmidt. Then
\begin{eqnarray*}
\Delta(1_{\mathcal{O}_1}*1_{\mathcal{O}_2})=\Delta(\sum\limits_{\lambda}g_{\mathcal{O}_2\mathcal{O}_1}^{\lambda}1_{\mathcal{O}_\lambda}) =\sum\limits_{\lambda}g_{\mathcal{O}_2\mathcal{O}_1}^{\lambda}\Delta(1_{\mathcal{O}_\lambda})\\
=\sum\limits_{\lambda} g_{\mathcal{O}_2\mathcal{O}_1}^{\lambda}(\sum\limits_{\alpha^{\prime},\beta^{\prime}} h_{\mathcal{O}_{\lambda}}^{\beta^{\prime}\alpha^{\prime}}1_{\mathcal{O}_{\alpha^{\prime}}}\otimes1_{\mathcal{O}_{\beta^{\prime}}}) =\sum\limits_{\alpha^{\prime},\beta^{\prime}}g_{\mathcal{O}_2\mathcal{O}_1}^{\alpha^{\prime}\oplus\beta^{\prime}} 1_{\mathcal{O}_{\alpha^{\prime}}}\otimes1_{\mathcal{O}_{\beta^{\prime}}},
\end{eqnarray*}

\begin{eqnarray*}
\Delta(1_{\mathcal{O}_1})*\Delta(1_{\mathcal{O}_2})=(\sum\limits_{\rho,\sigma}h_{\mathcal{O}_1}^{\sigma\rho} 1_{\mathcal{O}_{\rho}}\otimes1_{\mathcal{O}_{\sigma}})*(\sum\limits_{\epsilon,\tau}h_{\mathcal{O}_2}^{\tau\epsilon} 1_{\mathcal{O}_{\epsilon}}\otimes1_{\mathcal{O}_{\tau}})\\
=\sum\limits_{\rho,\sigma,\epsilon,\tau}h_{\mathcal{O}_1}^{\sigma\rho}h_{\mathcal{O}_2}^{\tau\epsilon} (1_{\mathcal{O}_{\rho}}*1_{\mathcal{O}_{\epsilon}})\otimes(1_{\mathcal{O}_{\sigma}}*1_{\mathcal{O}_{\tau}})\\
= \sum\limits_{\rho,\sigma,\epsilon,\tau}h_{\mathcal{O}_1}^{\sigma\rho}h_{\mathcal{O}_2}^{\tau\epsilon} (\sum\limits_{\alpha^{\prime},\beta^{\prime}}g_{\mathcal{O}_{\epsilon}\mathcal{O}_{\rho}}^{\alpha^{\prime}} g_{\mathcal{O}_{\tau}\mathcal{O}_{\sigma}}^{\beta^{\prime}}1_{\mathcal{O}_{\alpha^{\prime}}}\otimes 1_{\mathcal{O}_{\beta^{\prime}}})\\
=\sum\limits_{\alpha^{\prime},\beta^{\prime}}(\sum\limits_{\rho,\sigma,\epsilon,\tau}h_{\mathcal{O}_1}^{\sigma\rho} h_{\mathcal{O}_2}^{\tau\epsilon}g_{\mathcal{O}_{\epsilon}\mathcal{O}_{\rho}}^{\alpha^{\prime}} g_{\mathcal{O}_{\tau}\mathcal{O}_{\sigma}}^{\beta^{\prime}}1_{\mathcal{O}_{\alpha^{\prime}}}\otimes 1_{\mathcal{O}_{\beta^{\prime}}}).
\end{eqnarray*}

According to Theorem \ref{Greenthm}, it follows that
$$
\sum\limits_{\rho,\sigma,\epsilon,\tau}h_{\mathcal{O}_1}^{\sigma\rho} h_{\mathcal{O}_2}^{\tau\epsilon}g_{\mathcal{O}_{\epsilon}\mathcal{O}_{\rho}}^{\alpha^{\prime}} g_{\mathcal{O}_{\tau}\mathcal{O}_{\sigma}}^{\beta^{\prime}}=g_{\mathcal{O}_2\mathcal{O}_1}^{\alpha^{\prime}\oplus\beta^{\prime}}.
$$
Therefore $\Delta(1_{\mathcal{O}_1}*1_{\mathcal{O}_2})=\Delta(1_{\mathcal{O}_1})*\Delta(1_{\mathcal{O}_2})$. We have thus proved the theorem.
\end{proof}

\appendix
\section{Exact categories}

We recall the definition of an exact category (see \cite[Appendix A]{keller90}).
\begin{dfn}\label{def a1}
Let $\mathcal{A}$ be an additive category. A sequence
$$
X\xrightarrow{f}Y\xrightarrow{g}Z
$$
in $\mathcal{A}$ is called exact if $f$ is a kernel of $g$ and $g$ is a cokernel of $f$. The morphisms $f$ and $g$ are called inflation and deflation respectively. The short exact sequence is called a conflation. Let $\mathcal{S}$ be the collection of conflations closed under isomorphism and satisfying the following axioms

A0 $1_{0}:0\rightarrow0$ is a deflation.

A1 The composition of two deflations is a deflation.

A2 For every $h\in\Hom(X,X^{\prime})$ and every inflation $f\in\Hom(X,Y)$ in $\mathcal{A}$, there exists a pushout
\begin{equation*}
\xymatrix{
X \ar[r]^f \ar[d]_h & Y \ar[d]^{h^{\prime}} \\
X^{\prime} \ar[r]^{f^{\prime}} & Y^{\prime}
}
\end{equation*}
where $f^{\prime}\in\Hom(X^{\prime},Y^{\prime})$ is an inflation.

A3 For every $l\in\Hom(Z^{\prime},Z)$ and every deflation $g\in\Hom(Y,Z)$ in $\mathcal{A}$, there exists a pullback
\begin{equation*}
\xymatrix{
  Y^{\prime} \ar[r]^{g^{\prime}} \ar[d]_{l^{\prime}} & Z^{\prime} \ar[d]^{l} \\
  Y \ar[r]^{g} & Z   }
\end{equation*}
where $g^{\prime}\in\Hom(Y^{\prime},Z^{\prime})$ is an deflation.
Then $(\A,\mathcal{S})$ is called an exact category.
\end{dfn}

The definition of idempotent complete is taken from\cite[Definition 6.1]{buhler10}.
\begin{dfn}\label{def a3}
Let $\A$ be an additive category. The category $\A$ is idempotent complete if for every idempotent morphism $s:A\rightarrow A$ in $\A$, $s$ has a kernel $k:K\rightarrow A$ and a image $i:I\rightarrow A$ (a kernel of a cokernel of $s$) such that $A\cong K\oplus I$. We write $A\cong\text{Ker}s \oplus\text{Im}s$, for simplicity.
\end{dfn}

\end{document}